\documentclass[12pt]{amsart}
\usepackage{amsmath,amssymb,txfonts}
\usepackage{amssymb}
\usepackage{amsxtra}
\usepackage{amsthm, color}
\usepackage{txfonts}
\usepackage{graphicx}
\usepackage{times}

\usepackage{hyperref}
\usepackage[notref, notcite]{showkeys}
\usepackage{citeref}
\allowdisplaybreaks
\usepackage{color}
\usepackage{xcolor}
\usepackage{pgf,tikz,pgfplots}
\usepackage{mathrsfs}
\usetikzlibrary{arrows}

\hypersetup{
        colorlinks   = true,
        citecolor    = blue,
        linkcolor    = blue,
        urlcolor     = blue
}

\makeatletter
\@namedef{subjclassname@1991}{Subject}
\@namedef{subjclassname@2000}{Subject}
\@namedef{subjclassname@2010}{Subject}
\@namedef{subjclassname@2020}{{\rm 2020} Mathematics Subject Classification}
\makeatother

\def\rr{{\mathbb R}}
\def\rn{{{\rr}^n}}

\def\nn{{\mathbb N}}
\def\zz{{\mathbb Z}}

\def\BMO{{\mathop\mathrm{\,BMO\,}}}

\def\ls{\lesssim}
\def\gs{\gtrsim}

\def\dsum{\displaystyle\sum}

\def\r{\right}
\def\lf{\left}

\def\XXint#1#2#3{{\setbox0=\hbox{$#1{#2#3}{\int}$ }
\vcenter{\hbox{$#2#3$ }}\kern-.6\wd0}}

      \newcommand{\lt}[1]{[ {#1}] \lower.3ex\hbox{$_{t}$}}

\date{\today}

\newtheorem{Theorem}{Theorem}[section]
\newtheorem{theorem}{Theorem}[section]
\newtheorem{definition}[theorem]{Definition}

\newtheorem{remark}[theorem]{Remark}

\newtheorem{proposition}[theorem]{Proposition}
\newtheorem{lemma}[theorem]{Lemma}

\begin{document}
\title[Endpoint boundedness of Orlicz-BMO commutators]
{Endpoint boundedness of Orlicz-BMO commutators on Orlicz-Hardy type spaces}
	
\author{Zixing Zhuang}
\address{School of Mathematics and Statistics,
                       Fuzhou University,
                       Fuzhou, Fujian 350108, People's Republic of China}
\email{Zhuangzxmath@126.com}

\author{Chenglong Fang$^\ast$}
\address{School of Mathematics and Statistics,
                       Fuzhou University,
                       Fuzhou, Fujian 350108, People's Republic of China}
\email{clfang@fzu.edu.cn}

\thanks{Chenglong Fang is supported by the National Natural Science Foundation of China (\# 12501119) and the Fujian Provincial Natural Science Foundation Young Scientist Innovation Project (\# 2025J08035).}

\thanks{$^\ast$ Corresponding author}
	
\date{}
		
\subjclass[2020]{42B30; 42B35; 46E30; 47B47.}

%31C40 (1991-now) Fine potential theory; fine properties of sets and functions
%42B25  Maximal functions, Littlewood-Paley theory
%42B30 Hp-spaces
%42B35 (2000-now) Function spaces arising in harmonic analysis
% 26D10 Inequalities involving derivatives and differential and integral operators
%42B20;
%46E30  Spaces of measurable functions (L p-spaces, Orlicz spaces, K?the function spaces, Lorentz spaces, rearrangement invariant spaces, ideal spaces, etc.)
% 47B47  Commutators, derivations, elementary operators, etc
	
\keywords{Commutator; Orlicz-$\mathrm{BMO}$; Orlicz-Hardy;
Sublinear operator;  Boundedness.}

\begin{abstract}
Given a growth function $\varphi:[0,\infty)\rightarrow [0,\infty)$, it is established that the commutators generated by sublinear operators and Orlicz-$\mathrm{BMO}$ function $b$ are bounded from $H_{b}^{\varphi}(\mathbb{R}^{n})$ to $L^{1}(\mathbb{R}^{n})$, and from $H^{\varphi}(\mathbb{R}^{n})$ to $L^{1,\,\infty}(\rn)$, where
$H_{b}^{\varphi}(\mathbb{R}^{n})$ is a specific subspace of Orlicz-Hardy space $H^{\varphi}(\mathbb{R}^{n})$ and sublinear operators include Lusin area integral, g-function, Marcinkiewicz integral and Bochner-Riesz mean operator.
Under the assumptions $T^*1=0$ and $T^*b=0$,
it is shown that the Orlicz-$\mathrm{BMO}$ commutator associated with the Bochner-Riesz mean operator admits endpoint boundedness from $H_{b}^{\varphi}(\mathbb{R}^{n})$ to $H^{1}(\mathbb{R}^{n})$.
However, the commutators corresponding to other operators discussed in this paper do not possess the aforementioned endpoint boundedness,
and a counterexample is provided to illustrate this point.
\end{abstract}

%第一句话：四个次线性算子的交换子的两个有界性，其中此线性算子包括，，，，
%第二句话：当B-R算子满足T1和Tb条件时，其交换子从$H_{b}^{\varphi}(\mathbb{R}^{n})$到$H^{1}(\mathbb{R}^{n})$有界
%第三句话，但是But, 其他三个算子的交换子$H_{b}^{\varphi}(\mathbb{R}^{n})$到$H^{1}(\mathbb{R}^{n})$有界性不成立，本文给出了反例。

\maketitle	

\allowdisplaybreaks

%\tableofcontents

\section{Introduction}\label{s1}

\subsection{Background and question}\label{s1.1}
Given a locally integrable function $b$ and a sublinear operator $\mathfrak{T}$, the commutator $[b, \mathfrak{T}]$ is defined by
$$[b, \mathfrak{T}]h:=\mathfrak{T}((b-b(\cdot))h),$$
where $h$ is a continuous function with compact support. When $\mathfrak{T}$ is a linear operator, this definition reduces to the classical pointwise form
$$[b, \mathfrak{T}]h=b\mathfrak{T}h-\mathfrak{T}(bh).$$
Within the framework where $b\in\mathrm{BMO}(\rn)$ and $\mathfrak{T}=T$ is a
$\delta$-Calder\'{o}n-Zygmund operator. Paluszy\'{n}ski \cite[p.4, Remark]{P-1995} pointed out that the commutator $[b,T]$ fails to be bounded from $H^1(\rn)$ to $L^1(\rn)$.
To overcome this difficulty, P\'{e}rez \cite{Peter-1995} found a Hardy-type subspace $\mathcal{H}^{1}_{b}(\rn)\subset H^1(\rn)$ and showed
\begin{align}\label{H1b-1}
[b, T]: \ \mathcal{H}^{1}_{b}(\rn)\longrightarrow L^1(\rn).
\end{align}
Even so, $\mathcal{H}^{1}_{b}(\rn)$ does not achieve  the optimal subspace of $H^1(\rn)$ in \eqref{H1b-1}.
Via the bilinear decomposition technique of $H^{1}(\rn)\times\mathrm{BMO}(\mathbb{R}^{n})$ in \cite{Bonami2012}, Ky \cite{Ky-2013} found the largest subspace of $H^1(\rn)$, denoted by $H_b^{1}(\rn)$, and proved that
\begin{align}\label{H1b-2}
[b, T]: \ H_b^{1}(\rn)\longrightarrow L^1(\rn),
\end{align}
where $H_b^{1}(\rn)$ is the largest subspace of $H^1(\rn)$ such that, for any $\mathcal{X}\subset H^1(\rn)$, if $\mathcal{X}$ replaces $H_b^{1}(\rn)$ in \eqref{H1b-2} and satisfies its boundedness condition,
then $\mathcal{X}\subset H_b^{1}(\rn)$.
Moreover, under the assumptions $T^*1=0$ and $T^*b=0$,
Ky \cite[Theorem~3.6]{Ky-2013} proved that
\begin{align}\label{H1b-3}
[b, T]: \ H_b^{1}(\rn)\longrightarrow H^1(\rn).
\end{align}
In 2018, Liu, Yang and Yuan \cite{Liu-Yang-Yuan-2018} extended the endpoint boundedness in
\eqref{H1b-2} and \eqref{H1b-3} to spaces of homogeneous type.
For further contributions on this topic, we refer the reader to
\cite{A. Bonami-2023, Bonami-2019, X. Fu-2017,  X. Fu-D. Yang-2017, Ky2014, Liu-Yang-Yuan-2018-2, Liu-Fu-2017}.

A work extended Ky's results within the Orlicz space framework.
By the optimal bilinear decomposition of products of Orlicz-Hardy and Orlicz-$\mathrm{BMO}$
in \cite{F-L2024}, the authors \cite{F-L2024} constructed Orlicz-Hardy space $H_b^{\varphi}(\rn)$ and proved that the commutator generated by Orlicz function $b$ and $\delta$-Calder\'{o}n-Zygmund operator is bounded from the specific Orlicz Hardy space $H_b^{\varphi}(\rn)$ to $L^1(\rn)$,
and that under the cancellation condition $T^{*}1=T^{*}b=0$,
this commutator is also bounded from $H_b^{\varphi}(\rn)$ to $H^1(\rn)$.
Recently, the results related to the commutator generated by fractional integral operator are presented in \cite{Z-F-C-2026}.  Specifically,
it proved that the commutator associated with the fractional integral operator and Orlicz function $b$ is bounded from $H_b^{\varphi}(\rn)$ to $L^{\frac{n}{n-\alpha}}(\rn)$, and maps $H^{\varphi}(\rn)$ into $L^{\frac{n}{n-\alpha},\,\infty}(\rn)$, where $\alpha\in(0, n)$ and $\varphi:[0,\infty)\rightarrow [0,\infty)$ is a growth function.

Inspired by these developments, a natural question is whether the assertions derived in \eqref{H1b-2} and \eqref{H1b-3} remain valid for the commutators associated with the well-known sublinear operators under the Orlicz space framework, where the sublinear operators include Lusin area integral, g-function, Marcinkiewicz integral
and Bochner-Riesz mean operator.
This will be systematically discussed in this paper.
In order to state the main results, we give the corresponding definitions below.

A non-decreasing function $\varphi: [0, \infty)\rightarrow [0, \infty)$ is called an Orlicz function if it satisfies $\varphi(0)=0$, $\lim_{\tau\to\infty}\varphi(\tau)=\infty$ and $\varphi(\tau)>0$ for any $\tau\in(0, \infty)$.

\begin{definition}\label{def-1.1}
Let $\varphi: [0, \infty)\rightarrow [0, \infty)$ be an Orlicz function.
\begin{enumerate}
\item[\rm (i)]We say $\varphi$ is of lower type $q_1$ ($q_1\geq 0$) if there exists a constant $C_{q_1}>1$ such that
$$\varphi(st)\leq C_{q_1}s^{q_1}\varphi(t) \quad \text{for all} \ t\geq0 \ \text{and} \ s\in(0,1],$$
and $\varphi$ is of upper type $q_2$ ($q_2\geq 0$) if there exists a constant $C_{q_2}>1$ such that
$$\varphi(st)\leq C_{q_2}s^{q_2}\varphi(t)\quad \text{for all} \ t\geq0 \ \text{and} \ s\in[1, \infty).$$
\item[\rm (ii)]The indices
\begin{align*}
 i(\varphi):=\sup\{q_1\geq 0:\ \varphi\,\text{is of lower type}\, q_1\}
\end{align*}
and
\begin{align*}
 I(\varphi):=\inf\{q_2\geq 0:\ \varphi\,\text{is of upper type}\, q_2\}
\end{align*}
represent the critical lower and upper type indices of $\varphi$, respectively.
\end{enumerate}
\end{definition}\label{def-1.1}

\begin{definition}\label{def-1.2}
We call $\varphi: [0, \infty)\rightarrow [0,\infty)$ a growth function if and only if $\varphi$ is of upper type $1$ and of lower type $q$ for some $q\in(0,1]$.
\end{definition}

\begin{definition}\label{def-1.3}
For an Orlicz function $\varphi$, a locally integrable function $h$ on $\rn$ is a member of $\mathrm{BMO}_{\varphi}(\mathbb{R}^n)$ precisely if
\begin{align*}
 \|h\|_{\mathrm{BMO}_{\varphi}(\mathbb{R}^n)}:= \sup_{B \subset \mathbb{R}^n}\frac{1}{\|\mathbf{1}_B\|_{L^{\varphi}(\mathbb{R}^n)}}\int_B |h(x)-h_B|\,dx <\infty,
\end{align*}
where $h_B=\frac{1}{|B|}\int_{B}h(x)\,dx $ and the supremum is taken over all balls $B$ in $
\rn$. If $\varphi(t)=t$, then $\mathrm{BMO}_{\varphi}(\mathbb{R}^n)=\mathrm{BMO}(\mathbb{R}^n)$.
For more study and results about $\mathrm{BMO}$ type spaces, see \cite{F-L2023, F-L JFA2024, FLZ25, WL22}.
\end{definition}

A locally integrable function $h$ on $\rn$ is said to lie in the Orlicz-Lebesgue space $L^\varphi(\rn)$ if
$$\|h\|_{L^\varphi(\rn)}:=\inf\left\{\lambda>0: \ \int_{\rn}\varphi(|h(x)|/\lambda)\,dx\leq 1\right\}$$
is finite. Write $\mathcal{S}(\mathbb{R}^n)$ for the Schwartz class---the space of smooth functions whose derivatives all decay faster than any polynomial at infinity, and $\mathcal{S}'(\mathbb{R}^n)$ represents the space of tempered distributions. The auxiliary space
is defined by
$$\mathcal{S}_{0}(\mathbb{R}^n):=\left\{\phi\in \mathcal{S}(\mathbb{R}^n) : \
\|\phi\|= \sup_{\substack{x \in \mathbb{R}^n, \\ |\alpha|\leq 1}} (1+|x|)^{(n+1)} \, |\partial_x^\alpha \phi(x)| \leq 1 \right\}.$$

\begin{definition}\label{def-1.3}
Let $\phi\in\mathcal{S}_{0}(\mathbb{R}^n)$ and $\varphi$ be an Orlicz function.
\begin{enumerate}
\item[\rm (i)] Denote $\phi_t(\cdot)= t^{-n} \phi(\cdot/t)$ ($t>0$). For any $h\in\mathcal{S}'(\mathbb{R}^n)$, the nontangential grand maximal function $\mathfrak{M}h$ is defined by
\begin{align*}
\mathfrak{M}(h)(x):= \sup_{\phi \in S_0(\mathbb{R}^n)} \; \sup_{\substack{y\in\mathbb{R}^n , \\ |y-x| < t}} |(h * \phi_t)(y)|  \quad \text{for all} \ x\in\mathbb{R}^n.
\end{align*}
\item[\rm (ii)] A tempered distribution $h$ belongs to the Orlicz-Hardy space $H^{\varphi}(\rn)$ precisely when its grand maximal function $\mathfrak{M}(h)$ is an element of $L^{\varphi}(\rn)$. The (quasi-)norm on $H^{\varphi}(\rn)$ is inherited from $L^{\varphi}(\rn)$ via
    $$\|h\|_{H^{\varphi}(\rn)}:=\|\mathfrak{M}h\|_{L^{\varphi}(\rn)}.$$
\end{enumerate}
\end{definition}

\begin{definition}\label{def-1.4}
Let $\varphi$ be a growth function satisfying $0<i(\varphi)\leq 1$. For a fixed $b\in \mathrm{BMO}_{\varphi}(\mathbb{R}^n)$, the space $H_b^{\varphi}(\rn)$ is defined as the family of functions $h\in H^{\varphi}(\rn)$ for which the commutator
$$[b, \mathfrak{M}]h:=\mathfrak{M}\left((b-b(\cdot))h(\cdot)\right)$$
lies in $L^1(\rn)$. The norm associated with $H_b^{\varphi}(\rn)$ is thereby defined as
\begin{align*}
\|h\|_{H_b^\varphi(\mathbb{R}^n)}:= \|h\|_{H^\varphi(\mathbb{R}^n)} \|b\|_{\mathrm{BMO}_\varphi(\mathbb{R}^n)} + \|[b,\mathfrak{M}]\|_{L^1(\mathbb{R}^n)}.
\end{align*}
\end{definition}

The subsequent three sections present the definitions of the corresponding operators and the main results for each operator.

\subsection{Lusin area integral and g-function}\label{s1.2}
Let $P(\cdot)=(1+|\cdot|^{2})^{-2}$ be the Poisson kernel and $u_h(\cdot,  t):= h*P_t(\cdot)$ ($t>0$) be the Possion integral of $h$. The Lusin area integral $\mathbf{S}$ is defined as
$$\mathbf{S}(h)(x):=\left(\int_{\Gamma(x)}|\nabla u_h(y,\,t)|^{2}t^{1-n}\,dy\,dt\right)^{1/2} \quad \text{for all} \ x\in\mathbb{R}^n,$$
where $\Gamma(x):=\left\{(y, t)\in\rn\times(0,\infty): \, |y-x|<t\right\}$. And the $g$-function is defined as
$$\mathcal{G}(h)(x):=\left(\int_{0}^{\infty}\lf|\lf(\nabla P_{t}*h\r)(x)\r|^{2}t\,dt\right)^{1/2} \quad \text{for all} \ x\in\mathbb{R}^n.$$
Stein \cite[pp. 450-458]{Stein-1958} discussed the boundedness of $\mathbf{S}$ and $\mathcal{G}$. In detail, it is proved that
$$\mathbf{S} \ \text{and} \ \mathcal{G}: \ L^{q}(\rn)\rightarrow L^{q}(\rn) \quad
\text{for all} \ q\in(1, \infty)$$
and
$$\mathbf{S} \ \text{and} \ \mathcal{G}: \ L^{1}(\rn)\rightarrow L^{1,\,\infty}(\rn).$$
In 1972, Fefferman and Stein \cite[Theorem~9, Corollary~1 and Corollary~3]{Stein-Fefferman-1} established that Lusin area integral $\mathbf{S}$ and $g$-function $\mathcal{G}$ characterize $H^1(\rn)$, this is, if $\mathfrak{T}\in\{\mathbf{S}, \, \mathcal{G}\}$, then
\begin{align}\label{S,g-1.2}
h\in H^1(\rn) \ \text{if and only if} \ \mathfrak{T}h\in L^1(\rn).
\end{align}

As will be shown in the subsequent theorem, the commutators associated with Lusin area integral and $g$-function still hold for \eqref{H1b-2} under the Orlicz space setting, while \eqref{H1b-3} is not effective, and a counterexample is given in Theorem \ref{the-ex-4.1}.

\begin{theorem}\label{main-Thm-1}
Let $\varphi$ be a growth function with
\begin{align*}
n/(n+1)<i(\varphi)\leq I(\varphi)<1
\end{align*}
and
\begin{align}\label{eq.1.5}
\begin{cases}
1/I(\varphi)>1/i(\varphi)-1 \quad  &\text{as}\ n=1;\\
1/I(\varphi)>(\lfloor n(1/i(\varphi)-1)\rfloor+n-1)/n \quad  &\text{as}\ n\geq 2.
\end{cases}
\end{align}
Assume that $b\in\mathrm{BMO}_{\varphi}(\rn)$
and $\mathfrak{T}\in\{\mathbf{S}, \, \mathcal{G}\}$.
Then the following conclusions hold.
\begin{enumerate}
\item[\rm (i)]If $h\in H_b^{\varphi}(\rn)$, then there exists a positive constant $C$  such that
$$\|[b, \mathfrak{T}]h\|_{L^{1}(\rn)}\leq C \|h\|_{H_b^{\varphi}(\rn)}.$$
\item[\rm (ii)]If $h\in H^{\varphi}(\rn)$, then there exists a positive constant $C$  such that
    $$\|[b, \mathfrak{T}]h\|_{L^{1,\,\infty}(\rn)}\leq C \|h\|_{H^{\varphi}(\rn)}.$$
\end{enumerate}
\end{theorem}

\begin{remark}
The space $H_b^\varphi(\mathbb{R}^n)$ is the largest subspace of $H^\varphi(\mathbb{R}^n)$ for which the commutator $[b, \mathfrak{T}]$ in Theorem \ref{main-Thm-1}(i),
that is, if $X \subset H^\varphi(\mathbb{R}^n)$ is any subspace such that $[b, \mathfrak{T}]$  is bounded from $X$ to $L^{1}(\mathbb{R}^n)$, then $[b, \mathfrak{T}]h \in L^{1}(\mathbb{R}^n)$ for all $h\in X$. By this, \eqref{S,g-1.2} and \cite[Theorem~5.10]{F-L2024}, it follows that
$h \in H_b^\varphi(\mathbb{R}^n)$, hence $X \subseteq H_b^\varphi(\mathbb{R}^n)$.
\end{remark}

\subsection{Marcinkiewicz integral}\label{s1.2}
Let $\Omega\in L^1(S^{n-1})$ be a homogeneous function of degree zero on $\rn$ and satisfy
\begin{align}\label{eq.1.1}
\int_{S^{n-1}}\Omega(x')\,dx'=0,
\end{align}
where $x'=x/|x|$ for any $x\in\rn\setminus\{\vec{\mathbf{0}}\}$.
The Marcinkiewicz integral is firstly introduced by Stein \cite{Stein-1958}. For $\Omega\in L^1(S^{n-1})$, the Marcinkiewicz integral $\mu_\Omega$ is defined by
$$\mu_{\Omega}(h)(x):=\left(\int_{0}^{\infty}\left|\int_{|x-y|\leq t}\frac{\Omega(x-y)}{|x-y|^{n-1}}f(y)\,dy\right|^{2}\frac{dt}{t^3}\right)^{1/2}
\quad \text{for all} \ x\in\mathbb{R}^n.$$
Let $0<\alpha\leq 1$. A function $\Omega$ is said to belong to the class $\mathrm{Lip}_{\alpha}$ if it satisfies the Lipschitz condition of order $\alpha$,
namely,
\begin{align*}
|\Omega(x')-\Omega(y')|\leq C|x'-y'|^{\alpha} \quad \text{for all} \ x', y'\in S^{n-1}.
\end{align*}
If $\Omega\in L^1(S^{n-1})$ and $\Omega\in\mathrm{Lip}_{\alpha}$,
then Stein \cite{Stein-1958} proved that
$$\mu_\Omega: \ L^1(\rn)\rightarrow L^{1,\,\infty}(\rn)$$
and
$$\mu_\Omega: \ L^q(\rn)\rightarrow L^q(\rn)\quad \text{for all} \ q\in(1, 2].$$
Subsequently, Benedek et al. \cite[Theorem~1 and Theorem~3]{Benedek-1962} extended the above results to the case $1<p<\infty$ under the assumption $\Omega\in L^{1}(S^{n-1})$.

Recall that the $L^p$-Dini condition is introduced in \cite{Y. D-2003}.
The $p$ order integral modulus of continuity of $\Omega$ is given by
$$\omega_{p, \, \delta}(\Omega):=\sup_{|\rho|<\delta}\left(\int_{S^{n-1}}|\Omega(\rho x')-\Omega(x')|^{p} \,d\sigma(x')\right)^{1/p},$$
where $p\in[1, \infty)$, $\delta>0$, the rotation $\rho$ on $\rn$ satisfies $|\rho|=\|\rho-I\|$ and $d\sigma(x')$ means the surface measure on the unit sphere $S^{n-1}$.
For any $p\in[1, \infty)$ and $\beta\in[0,1)$,
the class $\mathrm{Din}^{p}_{\beta}$ consists of those functions $\Omega\in L^p(S^{n-1})$ that satisfy the $L^p$-Dini condition of order $\beta$, that is,
$$\int_{0}^{1}\frac{\omega_{p, \, \delta}(\Omega)}{\delta^{1+\beta}}\,d\delta <\infty.$$
In particular, for the endpoint case $p=\infty$, the corresponding $L^\infty$-Dini condition is defined by requiring the integrability condition
$\int_{0}^{1}\frac{\omega_{\infty, \, \delta}(\Omega)}{\delta^{1+\beta}}\,d\delta<\infty,$
where the modulus of continuity is given by
$$\omega_{\infty, \, \delta}(\Omega):=\sup_{|\rho|<\delta}\|\Omega(\rho x')-\Omega(x')\|_{L^{\infty}(S^{n-1})}.$$
In the special case $\beta=0$, the $L^p$-Dini condition of order $\beta$ naturally degenerates to the classical $L^p$-Dini condition widely adopted in the paper.

Under the assumptions that $\Omega$ satisfies \eqref{eq.1.1} and the $L^1$-Dini condition,
Ding, Lu and Xue \cite[Theorem~1]{Dingyong-2002} proved that
$$\mu_\Omega: \ H^1(\rn)\rightarrow L^{1}(\rn).$$
Let $b\in\mathrm{BMO}(\rn)$.  Zhang and Wu \cite{Z-Wu-2005} investigated the commutator $[b, \mu_\Omega]$ and rigorously established the following boundedness result:
$$[b, \mu_\Omega]: \ H^1(\rn)\rightarrow L^{1,\,\infty}(\rn).$$
For more about the study of $[b, \mu_\Omega]$, we refer the interested reader to \cite{Hyy-Whx-2023, Luguanghui-2021, Zuo-2007}.

The theorem to be presented below verifies that the commutator related to the Marcinkiewicz  integral still holds for \eqref{H1b-2} in the Orlicz framework, while \eqref{H1b-3} is invalid, and a concrete counterexample is given in Theorem \ref{the-ex-4.1}.

\begin{theorem}\label{Theo-main-b-2}
Let $\Omega\in\mathrm{Din}^{p}_{\beta}$ with $p\in(1,\,\infty]$ and $\beta\in(0,1)$.
Denote
$$\epsilon=\min\lf\{\beta, \ \frac{1}{2}\left(1-\frac{1}{p}\right)\r\}.$$
Assume that  $\varphi$ is a growth function  satisfying $n/(n+\epsilon)<i(\varphi)\leq I(\varphi)<1$ and \eqref{eq.1.5}. Let $b\in\mathrm{BMO}_{\varphi}(\rn)$.
Then the following conclusions hold.
\begin{enumerate}
\item[\rm (i)] If $h\in H_b^{\varphi}(\rn)$, then there exists a positive constant $C$  such that
$$\|[b, \mathfrak{\mu_\Omega}]h\|_{L^{1}(\rn)}\leq C \|h\|_{H_b^{\varphi}(\rn)}.$$
\item[\rm (ii)] If $h\in H^{\varphi}(\rn)$, then there exists a positive constant $C$  such that
    $$\|[b, \mathfrak{\mu_\Omega}]h\|_{L^{1,\,\infty}(\rn)}\leq C \|h\|_{H^{\varphi}(\rn)}.$$
\end{enumerate}
\end{theorem}

\begin{remark}
Let $p\in(1,\,\infty]$ and $\beta\in(0,1)$. From \cite[p.182]{Hyy-Whx-2023}), it follows that
$$\mathrm{Lip}_{\beta}\subset \mathrm{Din}^{p}_{\beta},$$
which implies that Theorem \ref{Theo-main-b-2} remains valid if the condition $\mathrm{Din}^{p}_{\beta}$ is replaced by $\mathrm{Lip}_{\beta}$.
\end{remark}

\subsection{Bochner-Riesz mean operator}\label{s1.3}
For $\delta>(n-1)/2$ and $0<R<\infty$, given test function $h$ on $\rn$, the Bochner-Riesz mean operator $T^{\delta}_{R}$ is defined by
$$T_R^{\delta}(h)(x):=\int_{|\xi|<R}\hat{h}(\xi)
\left(1-\frac{|\xi|^{2}}{R^2}\right)^{\delta}e^{2\pi ix\cdot\xi}\,d\xi$$
and the Bochner-Riesz maximal operator $T^{\delta}_{*}$ is denoted as
$$T_{*}^{\delta}(h)(x)=\sup_{R>0}\left|T_R^{\delta}h(x)\right| \quad \text{for all} \  x\in\rn.$$
The Bochner-Riesz mean operator admits  the following convolution representation (see \cite{Lee-2006}):
$$T_R^{\delta}(h)(x)= (h*\phi_{1/R})(x),$$
where the kernel satisfies the scaling relation $\phi_{\epsilon}(x)=\epsilon^{-n}\phi(x/\epsilon)$ with $$\phi(x)=\left\{(1-|\cdot|^{2})^{\delta}_{+}\right\}^\wedge(x).$$
Moreover, it follows from \cite{Sato-1995} that,  for any $0<p<1$ and $\delta=n/p-(n+1)/2$, the kernel obeys the estimate
\begin{align}\label{Sato-2}
\sup_{x\in\rn}(1+|x|)^{n/p}|D^\alpha\phi(x)|\leq C_{\alpha,n,p}\quad\text{for every multi-index}\,\alpha.
\end{align}

Let $\delta>(n-1)/2$.
A classical result (see \cite[p.389]{Stein-Ana-2}) showed that for any $p\in[1, \infty]$,
$$T_R^{\delta} \ \text{and} \ T_{*}^{\delta}: \ L^p(\rn)\longrightarrow L^p(\rn).$$
Moreover, Sj\'{o}lin \cite{P.-1981}, Stein, Taibleson and Weiss \cite{Stein-Taibleson-1981} proved that
$$T_R^{\delta}: \ H^{p}(\rn)\longrightarrow H^{p}(\rn) \quad \text{for all} \ 0<p\leq 1.$$
By this and $H^1(\rn)\subset L^1(\rn)$, we conclude that
$$T_R^{\delta}: \ H^{1}(\rn)\longrightarrow L^{1}(\rn).$$
In addition, Christ and Michael \cite{Christ-Michael-1988} proved that $T_R^{\delta}$ is of weak type $(1,1)$ for $\delta=(n-1)/2$ and $T_R^{\delta}$ fails to be bounded on $L^1(\rn)$ for any $\delta\leq (n-1)/2$.
There are many interesting results about the commutators of Bochner-Riesz mean operator and Bochner-Riesz maximal operator, we refer the interested readers to \cite{L. Z. Liu-2006, L. Liu and Q. Tong-2005, Z. Liu-2003}.

To present the main theorem in this subsection, we below recall the definitions of $L^{q}_{B}$ and $(\varphi, q, s)$-atom.

\begin{definition}\label{LqB}
For a ball $B\subset\mathbb{R}^n$ and exponent $q\in[1,\infty]$,
we write $L^{q}_{B}$ for the function space formed by all measurable functions $h$ on $\mathbb{R}^n$ with $\operatorname{supp}h\subset B$ and finite norm, whose norm is given by
\begin{align*}
\|h\|_{L^{q}_{B}}:=
\begin{cases}
\displaystyle\left(\frac{1}{|B|}\int_B |h(x)|^{q}\,dx\right)^{1/q}<\infty, & 1\le q<\infty;\\[2.5ex]
\displaystyle\operatorname{ess\,sup}_{x\in B}|h(x)|<\infty, & q=\infty.
\end{cases}
\end{align*}
\end{definition}

\begin{definition}\label{atom}
Let $\varphi$ be a growth function with critical lower type exponent $i(\varphi)$. A triple $(\varphi, q, s)$ is defined as admissible if $q\in(1,\infty]$ and the positive integer $s$ satisfies
$$s\geq m(\varphi):=\lfloor n(1/i(\varphi)-1)\rfloor.$$
A locally integrable function $a$ on $\rn$ is called a $(\varphi, q, s)$-atom when the following three conditions hold:
\begin{enumerate}
\item[\rm (i)] $\operatorname{supp} a \subset B$ for some ball $B\subset\mathbb{R}^n$;
\item[\rm (ii)] $\left\|a\right\|_{L^{q}_{B}}\leq\left\|{\mathbf 1}_{B}\right\|^{-1}_{L^{\varphi}(\rn)}$;
\item[\rm (iii)] $\int_{\mathbb{R}^{n}}x^{\beta}a(x)\,dx=0$  for every multi-index $\beta\in\zz_+^n$ with $|\beta|\leq s$.
\end{enumerate}
\end{definition}

The following theorem states that \eqref{H1b-2} and \eqref{H1b-3} still hold for commutators of Bochner-Riesz mean operator $T_R^{\delta}$ in the Orlicz setting.

\begin{theorem}\label{Theo-main-b-3}
Suppose that $\delta>(n-1)/2$ and  $\delta_0$ satisfies
$$(n-1)/2<\delta_0<\min\{\delta, \, (n+1)/2\}.$$
Let $\varphi$ be a growth function satisfying \eqref{eq.1.5} and $$2n/(n+1+2\delta_0)<i(\varphi)\leq I(\varphi)<1.$$
Assume that $b\in\mathrm{BMO}_{\varphi}(\rn)$ and $\mathfrak{T}\in\{T_R^{\delta}, \, T_*^{\delta}\}$ with $0<R<\infty$. Then the following conclusions hold.
\begin{enumerate}
\item[\rm (i)]If $h\in H_b^{\varphi}(\rn)$, then there exists a positive constant $C$  such that
$$\|[b, \mathfrak{T}]h\|_{L^{1}(\rn)}\leq C \|h\|_{H_b^{\varphi}(\rn)}.$$

\item[\rm (ii)]If $h\in H^{\varphi}(\rn)$, then there exists a positive constant $C$  such that
    $$\|[b, \mathfrak{T}]h\|_{L^{1,\,\infty}(\rn)}\leq C \|h\|_{H^{\varphi}(\rn)}.$$

\item[\rm (iii)] Let $h\in H_b^{\varphi}(\rn)$ and $a_{0}$ be a $(\varphi, q, 0)$-atom with $q\in(1, \infty)$.
If $T_R^{\delta}$ satisfies the $T^*1$ condition and the $T^*b$ condition, that is,
\begin{align}\label{eq.2.36}
\int_{\rn}T_R^{\delta}(a_{0})(x)\,dx=0=\int_{\rn}b(x)T_R^{\delta}(a_{0})(x)\,dx,
\end{align}
then there exists a positive constant $C$ such that
$$\|[b, T_R^{\delta}]h\|_{H^{1}(\rn)}\leq C \|h\|_{H_b^{\varphi}(\rn)}.$$
\end{enumerate}
\end{theorem}

\subsection{Paper structure}\label{s1.6}

%The organization of the article as follows.
In Section \ref{s2}, we give the definition of $\mathcal{K}_{\varphi}$ and prove that all of the aforementioned operators belong to $\mathcal{K}_{\varphi}$. Then we prove Theorems \ref{main-Thm-1}, \ref{Theo-main-b-2} and \ref{Theo-main-b-3}(i)-(ii) in Section \ref{Sec-3}. Finally,
we finish the proof of Theorem \ref{Theo-main-b-3}(iii),
and illustrate via a counterexample  why only the Bochner-Riesz mean operator $T_R^{\delta}$ possesses the boundedness stated in Theorem \ref{Theo-main-b-3}(iii).

\medskip

\noindent{\bf Notation.\,}
%Throughout the whole article, we adopt the following notation and notions.
We introduce the following notation and notions used uniformly in this work.
Let $\nn:=\{1,2,\cdots\}$. For any real number $s$, the symbol
$\lfloor s\rfloor$ denotes the greatest integer $i$ satisfying $i\leq s$. We denote by $d\sigma(x')$ the surface measure on the unit sphere $S^{n-1}$. Let $\vec{\mathbf{0}}$ denote the origin (or the zero vector) in $\rn$. For a function $u(x,t)$, define its extended gradient as the column vector $\nabla_{u}:=\left(\frac{\partial u}{\partial x_{1}},\, \dots\,,\frac{\partial u}{\partial x_{n}},\, \frac{\partial u}{\partial t}\right)$, and $|\nabla_{u}|^{2}:=|\nabla_{x}u|^{2}+|\partial_{t}u|^{2}$, where $\partial_{t}:=\frac{\partial_{u}}{\partial_{t}}$. For a function $h$, $\hat{h}$ denotes its Fourier transform. For any $a\in R$, its positive part is defined by $a_{+}:=\max\{a,\,0\}$. Let $C^{\infty}(\rn)$ be the space of all infinitely differentiable functions on $\rn$ and the symbol $C_{c}^{\infty}(\rn)$ denotes the subspace of such functions with compact support. The symbol $h_{1}\ls h_{2}$ (or $h_{2}\gs h_{1})$ means $h_{1}\leq Ch_{2}$ (or $h_{2}\geq Ch_{1}$), and $h_{1}\approx h_{2}$ means both $h_{1}\ls h_{2}$ and $h_{2}\ls h_{1}$ hold.

\section{Class $\mathcal{K}_{\varphi}$ and related examples}\label{s2}
The primary purpose of this section is to define class $\mathcal{K}_{\varphi}$ and verify that Lusin area integral, $g$-function, Marcinkiewicz integral and Bochner-Riesz mean operator belong to class $\mathcal{K}_{\varphi}$ in Propositions \ref{proposition-Larea-S}, \ref{proposition-g-func}, \ref{proposition-mar} and \ref{proposition-BR}, respectively.
To achieve this, we present the following definition and lemma.

\begin{definition}\label{def2.1}
Given a growth function $\varphi$ with $0<i(\varphi)\leq 1$,
we denote by $\mathcal{K}_{\varphi}$ the set of all sublinear operators $\mathfrak{T}$ obeying the three properties below:
\begin{enumerate}
\item[\rm (i)] $\mathfrak{T}$ is bounded from $H^{1}(\mathbb{R}^{n})$ to $L^{1}(\rn)$;
\item[\rm (ii)]$\mathfrak{T}$  is bounded from $L^{1}(\mathbb{R}^{n})$ to
$L^{1, \, \infty}(\rn)$;
\item[\rm (iii)] For any $b\in\mathrm{BMO}_{\varphi}(\mathbb{R}^{n})$ and any $(\varphi,q,0)$-atom
$a$ supported in ball $B$ with $q\in(1,\infty)$, some constant $C>0$ exists such that
\begin{align*}
\|(b-b_{B})\mathfrak{T}a\|_{L^{1}(\mathbb{R}^{n})}
\leq C\|a\|_{L^{q}_{B}}\|{\mathbf{1}_{B}\|_{L^{\varphi}(\rn)}\|b\|_{\mathrm{BMO}_{\varphi}(\rn)}}.
\end{align*}
\end{enumerate}
\end{definition}

\begin{lemma}\label{lemma2.5}
Given $b\in\mathrm{BMO}_{\varphi}(\rn)$ and a $(\varphi,q,0)$-atom $a$ with $\operatorname{supp}a \subset B$ for some ball $B$ and $q\in(1,\infty)$.
For a sublinear operator $\mathfrak{T}$ that is bounded on $L^{q}(\rn)$,
the following estimate holds:
\begin{align*}
\|(b-b_B)\mathbf{1}_{2B} \mathfrak{T} (a)\|_{L^{1}(\rn)}\ls \|a\|_{L^{q}_{B}}\|{\mathbf{1}_{B}\|_{L^{\varphi}(\rn)}\|b\|_{\mathrm{BMO}_{\varphi}(\rn)}}.
\end{align*}
\end{lemma}
\begin{proof}
This result has been rigorously established in the existing literature, thus we omit the proof and refer interested readers to \cite[Lemma~5.6]{F-L2024} for the full presentation.
\end{proof}

The following four propositions will respectively demonstrate that the four sublinear operators, namely area integral operator $\mathbf{S}$,  $g$-function $\mathcal{G}$,  Marcinkiewicz integral $\mu_\Omega$ and the Bochner-Riesz mean operator $T_R^{\delta}$, belong to the class $\mathcal{K}_{\varphi}$.

\begin{proposition}\label{proposition-Larea-S}
Let $\varphi$ be a growth function with $n/(n+1)<i(\varphi)\leq I(\varphi)<1$.
Then area integral operator $\mathbf{S}\in \mathcal{K}_{\varphi}$.
\end{proposition}
\begin{proof}
It follows from \cite[Corollary~1]{Stein-Fefferman-1} that the area integral operator $\mathbf{S}$ is bounded from $H^{1}(\rn)$ to $L^{1}(\rn)$. The boundedness of $\mathbf{S}$ from $L^{1}(\rn)$ to $L^{1,\,\infty}(\rn)$ is a classical result established in \cite[Theorem~3]{Stein-1958}. Moreover, the $L^{q}(\rn)$ boundedness of $\mathbf{S}$ for all $q\in(1,\infty)$ is guaranteed by \cite[Theorem~4]{Stein-1958}. In view of these existing boundedness properties, it therefore suffices to verify that
$\mathbf{S}$ satisfies Definition \ref{def2.1}(iii).

Now, we verify that $\mathbf{S}$ satisfies Definition \ref{def2.1}(iii). Suppose that $b\in\mathrm{BMO}_{\varphi}(\rn)$ and $a$ is a $(\varphi,q,0)$-atom with $\operatorname{supp}a \subset B$ and $q\in(1,\infty)$.
Based on the $L^{q}(\rn)$ boundedness of $\mathbf{S}$,
we use Lemma \ref{lemma2.5} to derive
\begin{align*}
\|(b-b_{B})\mathbf{S} (a)\|_{L^{1}(\mathbb{R}^{n})}
&\leq\|(b-b_B)\mathbf{1}_{2B} \mathbf{S} (a)\|_{L^{1}(\rn)}
+ \|(b-b_B)\mathbf{1}_{(2B)^c} \mathbf{S}(a)\|_{L^{1}(\rn)}\\
&\ls \|a\|_{L^{q}_{B}}\|{\mathbf{1}_{B}\|_{L^{\varphi}(\rn)}\|b\|_{\mathrm{BMO}_{\varphi}(\rn)}}
+\|(b-b_B)\mathbf{1}_{(2B)^c} \mathbf{S}(a)\|_{L^{1}(\rn)}.
\end{align*}
Therefore, we only need to prove that
\begin{align}\label{2ec-S}
\|(b-b_B)\mathbf{1}_{(2B)^c} \mathbf{S}(a)\|_{L^{1}(\rn)}\ls \|a\|_{L^{q}_{B}}\|{\mathbf{1}_{B}\|_{L^{\varphi}(\rn)}\|b\|_{\mathrm{BMO}_{\varphi}(\rn)}}.
\end{align}

To estimate \eqref{2ec-S}, we denote $B=B(x_0, r)$ and set $x\in(2^{k+1}B)\setminus  (2^{k}B)$ for any $k\in\nn$.
Write $\partial_j=\frac{\partial}{\partial_{x_j}}$($j\in\nn$ and $1\leq j\leq n$) and $\partial_{n+1}=\frac{\partial}{\partial_{t}}$. Choose $\sigma$ such that $\varphi$ is of lower type $\sigma$ with $n/(n+1)<\sigma<i(\varphi)$. So,
\begin{align}\label{eq-2.2}
n\left(\frac{1}{\sigma}-1\right)<1.
\end{align}
By the definition of $\mathbf{S}$, we see that for any $x\in(2^{k+1}B)\setminus  (2^{k}B)$,
\begin{align}\label{eq.2.4}
[\mathbf{S}(a)(x)]^{2}&= \int_{\Gamma(x)}|\nabla u_a(y,t)|^{2}t^{1-n}\,dy\,dt\\
&= \int_{\Gamma(x)}\left[\sum_{j=1}^n \left| \frac{\partial}{\partial x_j}(P_t*a)(y)\right|^{2}+ \left|\frac{\partial}{\partial t}(P_t*a)(y)\right|^{2}\right]t^{1-n}\,dy\,dt\notag\\
&= \sum_{j=1}^{n+1}\int_{\Gamma(x)}\lf|\lf((\partial_j P_t)*a\r)(y)\r|^{2}t^{1-n}\,dy\,dt\notag,
\end{align}
where $\Gamma(x):=\left\{(y,t)\in\rn\times(0,\infty):|y-x|<t\right\}$.
Using $\operatorname{supp}a \subset B$ and $\int_{\rn}a(z)\,dz=0$, we deduce that for any $t\in(0,\infty)$ and $y\in\rn$,
\begin{align*}
\lf|\lf((\partial_j P_t)*a\r)(y)\r|
&=\left|\int_{B}\partial_{j}P_t(y-z)a(z)\,dz\right|\\
&=\left|\int_{B}(\partial_{j}P_t(y-z)-\partial_{j}P_t(y-x_0))a(z)\,dz\right|.
\end{align*}
Further, for $j\in\{1,2,\dots,n+1\}$, there exist $\theta\in(0,1)$ and $\xi_{\theta}={\theta}z+(1-{\theta})x_0$ with $z\in B(x_0, r)$
such that
\begin{align*}
\lf|\partial_{j}P_t(y-z)-\partial_{j}P_t(y-x_0)\r|
&\leq\sum_{i=1}^{n}\lf|\partial_{i}\partial_{j}P_t(y-({\theta}z+(1-{\theta})x_0))\r||x_0-z|\\
&\ls \frac{n}{(t^{2}+\lf|y-\xi_{\theta}\r|^{2})^{(n+2)/2}}|x_0-z|,
\end{align*}
where the last step used \cite[(11.9)]{Yang-2017}, that is,
$\partial_{i}\partial_{j}P_t(y-\xi_{\theta})
\ls\frac{1}{(t^{2}+\lf|y-\xi_{\theta}\r|^{2})^{(n+2)/2}}.$
Substituting the above two estimates into \eqref{eq.2.4}, we obtain
\begin{align*}
[\mathbf{S}(a)(x)]^{2}
&\ls\sum_{j=1}^{n+1}\int_{\Gamma(x)}
\left(\int_{B}\frac{n}{(t^{2}+\lf|y-\xi_{\theta}\r|^{2})^{(n+2)/2}}|x_0-z||a(z)|\,dz\right)^{2}
t^{1-n}\,dy\,dt\\
&\ls\sum_{j=1}^{n+1}\int_{0<t<{d/4}}\int_{|y-x|<t}
\left(\int_{B}\frac{n}{(t^{2}+\lf|y-\xi_{\theta}\r|^{2})^{(n+2)/2}}|x_0-z||a(z)|\,dz\right)^{2}
t^{1-n}\,dy\,dt\\
&\qquad+\sum_{j=1}^{n+1}\int_{{d/4}\leq t<\infty}\int_{|y-x|<t}
\left(\int_{B}\frac{n}{(t^{2}+\lf|y-\xi_{\theta}\r|^{2})^{(n+2)/2}}|x_0-z||a(z)|\,dz\right)^{2}
t^{1-n}\,dy\,dt\\
&=:{\rm I}+{\rm II},
\end{align*}
where  $x\in\lf(2^{k+1}B(x_0, r)\r)\setminus\lf(2^{k}B(x_0, r)\r)$ and $d=|x-x_0|\geq 2r$.

To estimate ${\rm I}$, by $|y-x|<t<d/4=|x-x_0|/4$ and $\xi_{\theta}={\theta}z +(1-{\theta})x_{0}\in B(x_0, r)$ with $z\in B$, we have
\begin{align*}
|y-\xi_{\theta}|&\geq |y-x_0|-|x_0-\xi_{\theta}|\\
&\geq |x-x_0|-|x-y|-|x_0-\xi_{\theta}|\\
&\geq d-d/4-r \\
&\geq {3d}/4-d/2=d/4.
\end{align*}
This, along with
\begin{align}\label{Ba-BAQ}
\int_B|a(z)|\,dz\leq |B|^{1-1/q}\|a\|_{L^{q}(\rn)}\approx r^{n}\|a\|_{L^{q}_{B}},
\end{align}
yields
\begin{align*}
&\int_{B}\frac{n}{(t^{2}+\lf|y-\xi_{\theta}\r|^{2})^{(n+2)/2}}|x_0-z||a(z)|\,dz\\
&\quad\ls \int_B \frac{n}{|y-\xi_{\theta}|^{n+2}}|x_0-z||a(z)|\,dz\\
&\quad\ls nr(d/4)^{-n-2}\int_B|a(z)|\,dz\\
&\quad\ls nr^{n+1}(d/4)^{-n-2}\|a\|_{L^{q}_{B}},
\end{align*}
which implies
\begin{align*}
{\rm I}
&\ls
n^{2}r^{2n+2}(d/4)^{-2n-4}\|a\|_{L^{q}_{B}}^{2}\int_{0<t<{d/4}}\int_{|y-x|<t}t^{1-n}\,dy\,dt\\
&\ls
n^{2}r^{2n+2}(d/4)^{-2n-4}\|a\|_{L^{q}_{B}}^{2}\int_{0<t<{d/4}}t^{1-n}t^{n}\,dt\\
&\ls \lf(\frac{r}{d}\r)^{2n+2}\|a\|_{L^{q}_{B}}^{2}.
\end{align*}

To estimate ${\rm II}$, by $z\in B(x_0, r)$ and the following rough estimate:
$$\frac{1}{(t^{2}+\lf|y-\xi_{\theta}\r|^{2})^{(n+2)/2}}\leq t^{-n-2},$$
we have
\begin{align*}
{\rm II}
&\ls\sum_{j=1}^{n+1}\int_{t\geq {d/4}}\int_{|y-x|<t}\left(rt^{-n-2}\int_B|a(z)|\,dz\right)^{2}t^{1-n}\,dy\,dt\\
&\ls\sum_{j=1}^{n+1}\int_{t\geq {d/4}}\int_{|y-x|<t}\left(r^{n+1}t^{-n-2}\|a\|_{L^{q}_{B}}\right)^{2}t^{1-n}\,dy\,dt\\
&\approx r^{2n+2}\|a\|_{L^{q}_{B}}^{2}\int_{t\geq {d/4}}t^{-(2n+3)}\,dt\\
&\ls\lf(\frac{r}{d}\r)^{2n+2}\|a\|_{L^{q}_{B}}^{2}.
\end{align*}
Finally,  combining the estimates of ${\rm I}$ and ${\rm II}$, we conclude that 
\begin{align*}
[\mathbf{S}(a)(x)]^{2}\ls {\rm I}+{\rm II}
\ls \lf(\frac{r}{d}\r)^{2n+2}\|a\|_{L^{q}_{B}}^{2} \quad \text{for all} \ x\in\lf(2^{k+1}B(x_0, r)\r)\setminus\lf(2^{k}B(x_0, r)\r).
\end{align*}
From this and $d=|x-x_0|\geq 2^{k}r$, it follows that
\begin{align*}
\mathbf{S}(a)(x)\ls \left(\frac{r}{|x-x_0|}\right)^{n+1}\|a\|_{L^{q}_{B}}
\ls 2^{-k(n+1)}\|a\|_{L^{q}_{B}},
\end{align*}
which implies
\begin{align*}
\|(b-b_B)\mathbf{1}_{(2B)^c} \mathbf{S}(a)\|_{L^{1}(\rn)}&\leq \sum_{k=1}^{\infty}\int_{(2^{k+1}B)\setminus(2^{k}B)}|b(x)-b_B||\mathbf{S}(a)(x)|\,dx\\
&\ls \|a\|_{L^{q}_{B}}\sum_{k=1}^{\infty}2^{-k(n+1)}\int_{2^{k+1}B}\left(|b(x)-b_{2^{k+1}B}|+
|b_B-b_{2^{k+1}B}|\right)\,dx.
\end{align*}
Note that $\varphi$ is of lower type $\sigma$. From \cite[Lemma~3.4 and Lemma~5.5]{F-L2024},
it follows that for any ball $\bar{B}\subset\rn$ and $p\in[1, \infty)$,
\begin{align}\label{eq.B-1}
\left(\int_{\bar{B}}|b(x)-b_{\bar{B}}|^{p}\,dx\right)^{1/p}
\ls |\bar{B}|^{1/p-1}\|\mathbf{1}_{\bar{B}}\|_{L^{\varphi}(\rn)}
\|b\|_{\mathrm{BMO}_{\varphi}(\rn)}
\end{align}
and 
\begin{align}\label{eq.B-2}
|b_{2^{i}\bar{B}}-b_{\bar{B}}|
\ls 2^{in/\sigma}
|2^{i}\bar{B}|^{-1}\|{\mathbf 1}_{\bar{B}}\|_{L^{\varphi}(\mathbb{R}^n)}
\|b\|_{\mathrm{BMO}_{\varphi}(\mathbb{R}^n)} \quad \text{for all} \ i\in\nn.
\end{align}
Synthesizing the three estimates obtained above, we apply \eqref{eq-2.2} to derive
\begin{align}\label{eq.2.18}
&\|(b-b_B)\mathbf{1}_{(2B)^c} \mathbf{S}(a)\|_{L^{1}(\rn)}\\
&\quad\ls \|a\|_{L^{q}_{B}}\sum_{k=1}^{\infty}2^{-k(n+1)}
\lf[\int_{2^{k+1}B}|b(x)-b_{2^{k+1}B}|
\,dx+|2^{k+1}B||b_B-b_{2^{k+1}B}|\r]\notag\\
&\quad\ls \|a\|_{L^{q}_{B}}\sum_{k=1}^{\infty}2^{-k(n+1)}\left(\|{\mathbf 1}_{2^{k+1}B}\|_{L^{\varphi}(\mathbb{R}^n)}\|b\|_{\mathrm{BMO}_{\varphi}(\rn)}+
2^{\frac{(k+1)n}{\sigma}}\|{\mathbf 1}_{B}\|_{L^{\varphi}(\mathbb{R}^n)}\|b\|_{\mathrm{BMO}_{\varphi}(\rn)}\right)\notag\\
&\quad\ls \|a\|_{L^{q}_{B}}\|{\mathbf 1}_{B}\|_{L^{\varphi}(\mathbb{R}^n)}
\|b\|_{\mathrm{BMO}_{\varphi}(\rn)}\sum_{k=1}^{\infty}2^{-k(n+1)}2^{\frac{(k+1)n}{\sigma}}\notag\\
&\quad \approx \|a\|_{L^{q}_{B}}\|{\mathbf 1}_{B}\|_{L^{\varphi}(\mathbb{R}^n)}
\|b\|_{\mathrm{BMO}_{\varphi}(\rn)}\sum_{k=1}^{\infty}2^{-k\lf(1-n\left(\frac{1}{\sigma} - 1\right)\r)}\notag\\
&\quad\ls \|a\|_{L^{q}_{B}}\|{\mathbf 1}_{B}\|_{L^{\varphi}(\mathbb{R}^n)}
\|b\|_{\mathrm{BMO}_{\varphi}(\rn)}\notag,
\end{align}
where the three step used \cite[(2.3)]{F-L2024}, that is,
\begin{align}\label{eq.simga-k1}
\|{\mathbf 1}_{2^{k+1}B}\|_{L^{\varphi}(\mathbb{R}^n)}
\ls 2^{\frac{(k+1)n}{\sigma}}\|{\mathbf 1}_{B}\|_{L^{\varphi}(\mathbb{R}^n)}.
\end{align}
This allows us to obtain that \eqref{2ec-S} holds. So, $\mathbf{S}$ satisfies Definition \ref{def2.1}(iii), as desired.
\end{proof}

\begin{proposition}\label{proposition-g-func}
Let $\varphi$ be a growth function with $n/(n+1)<i(\varphi)\leq I(\varphi)<1$.
Then $g$-function $\mathcal{G}$ belongs to $\mathcal{K}_{\varphi}$.
\end{proposition}
\begin{proof}
By virtue of the results in \cite[Corollary~3]{Stein-Fefferman-1} and \cite[Corollary and Theorem~5]{Stein-1958}, $g$-function $\mathcal{G}$ is known to be bounded from $H^{1}(\rn)$ to $L^{1}(\rn)$, from $L^{1}(\rn)$ to $L^{1,\,\infty}(\rn)$, and on $L^{q}(\rn)$ for every $q\in(1,\infty)$, respectively. This implies that the only remaining task is to confirm that $g$-function $\mathcal{G}$ fulfills Definition \ref{def2.1}(iii).
In fact, by Lemma \ref{lemma2.5}, we find that
\begin{align*}
\lf\|(b-b_{B})\mathcal{G}(a)\r\|_{L^{1}(\mathbb{R}^{n})}
&\ls \lf\|(b-b_B)\mathbf{1}_{2B}\mathcal{G}(a)\r\|_{L^{1}(\rn)}
+\lf\|(b-b_B)\mathbf{1}_{(2B)^c}\mathcal{G}(a)\r\|_{L^{1}(\rn)}\\
&\ls \|a\|_{L^{q}_{B}}\|{\mathbf{1}_{B}\|_{L^{\varphi}(\rn)}
\|b\|_{\mathrm{BMO}_{\varphi}(\rn)}}
+\lf\|(b-b_B)\mathbf{1}_{(2B)^c}\mathcal{G}(a)\r\|_{L^{1}(\rn)},
\end{align*}
where $a$ is a $(\varphi,q,0)$-atom with $\operatorname{supp}a \subset B$
and $q\in(1,\infty)$. So, it suffices to prove that
\begin{align}\label{G-b-bB-A}
\lf\|(b-b_B)\mathbf{1}_{(2B)^c}\mathcal{G}(a)\r\|_{L^{1}(\rn)}\ls \|a\|_{L^{q}_{B}}\|{\mathbf{1}_{B}\|_{L^{\varphi}(\rn)}\|b\|_{\mathrm{BMO}_{\varphi}(\rn)}}.
\end{align}

To consider \eqref{G-b-bB-A},
we denote  $\partial_j=\frac{\partial}{\partial_{x_j}}$($j\in\nn$ and $1\leq j\leq n$) and $\partial_{n+1}=\frac{\partial}{\partial_{t}}$. Moreover, we choose $\sigma$ such that $\varphi$ is of lower type $\sigma$ with $n/(n+1)<\sigma<i(\varphi)$.
Fix $B=B(x_0, r)$. Then, for any $x\in(2^{k+1}B)\setminus  (2^{k}B)$ with $k\in\nn$,
\begin{align*}
[\mathcal{G}(a)(x)]^{2}&= \int_{0}^{\infty}|\nabla P_t*a(x)|^{2}t\,dt\\
&= \int_{0}^{\infty}\left[\sum_{j=1}^n \left| \frac{\partial}{\partial x_j}(P_t*a)(x)\right|^{2}+ \left|\frac{\partial}{\partial t}(P_t*a)(x)\right|^{2}\right]t\,dt\\
&= \sum_{j=1}^{n+1}\int_{0}^{\infty}|(\partial_j P_t)*a(x)|^{2}t\,dt\\
&= \sum_{j=1}^{n+1}\int_{0}^{|x-x_0|/4}|(\partial_j P_t)*a(x)|^{2}t\,dt+ \sum_{j=1}^{n+1}\int_{|x-x_0|/4}^{\infty}|(\partial_j P_t)*a(x)|^{2}t\,dt.
\end{align*}
Further, we again use
$\partial_{i}\partial_{j}P_t(y-\xi_{\theta})
\ls\frac{1}{(t^{2}+\lf|y-\xi_{\theta}\r|^{2})^{(n+2)/2}}$ (see \cite[(11.9)]{Yang-2017})
and obtain
\begin{align*}
|(\partial_j P_t)*a(x)|
&\ls \int_B \frac{1}{(t^{2}+|x-\xi_\theta'|^{2})^{(n+2)/2}}|x_0-y||a(y)|\,dy, 
\end{align*}
where $\xi_\theta'={\theta}y +(1-{\theta})x_{0}\in B(x_0, r)$ since $y\in B(x_0, r)$. Therefore,
\begin{align*}
[\mathcal{G}(a)(x)]^{2}
&\ls\sum_{j=1}^{n+1}\int_{0}^{|x-x_0|/4}\lf(\int_B \frac{1}{(t^{2}+|x-\xi_\theta'|^{2})^{(n+2)/2}}|x_0-y||a(y)|\,dy\r)^{2}t\,dt\\
&\qquad+ \sum_{j=1}^{n+1}\int_{|x-x_0|/4}^{\infty}\lf(\int_B \frac{1}{(t^{2}+|x-\xi_\theta'|^{2})^{(n+2)/2}}|x_0-y||a(y)|\,dy\r)^{2}t\,dt\\
&=:{\rm J}_{1}+{\rm J}_{2}.
\end{align*}
To estimate ${\rm J}_{1}$, by $\xi_\theta'={\theta}y +(1-{\theta})x_{0}\in B(x_0, r)$,
we see that for any $x\in(2B)^c$,
$$|x-\xi_\theta'|\geq |x-x_0|-|x_0-\xi_\theta'|\geq |x-x_0|-r\geq |x-x_0|-|x-x_0|/2=|x-x_0|/2.$$
Then, by \eqref{Ba-BAQ}, we have
\begin{align*}
{\rm J}_{1}
&\leq \sum_{j=1}^{n+1}\int_{0}^{|x-x_0|/4}\lf(\int_B \frac{1}{(|x-\xi_\theta'|^{2})^{(n+2)/2}}|x_0-y||a(y)|\,dy\r)^{2}t\,dt\\
&\ls \sum_{j=1}^{n+1}\int_{0}^{|x-x_0|/4}\lf(|x-x_0|^{-(n+2)}r \int_B |a(y)|\,dy\r)^{2}t\,dt\\
&\ls \sum_{j=1}^{n+1}\int_{0}^{|x-x_0|/4}\lf(|x-x_0|^{-(n+2)}r^{n+1}
\|a\|_{L^{q}_{B}}\r)^{2}t\,dt\\
&\ls (n+1)|x-x_0|^{-2(n+2)}r^{2n+2}\|a\|_{L^{q}_{B}}^{2}\int_{0}^{|x-x_0|/4}t\,dt\\
&\ls (n+1)\lf(\frac{r}{|x-x_0|}\r)^{2(n+1)}\|a\|_{L^{q}_{B}}^{2}.
\end{align*}
To estimate ${\rm J}_{2}$, using \eqref{Ba-BAQ}, it follows that
\begin{align*}
{\rm J}_{2}
&\leq \sum_{j=1}^{n+1}\int_{|x-x_0|/4}^{\infty}\lf(\int_B \frac{1}{t^{n+2}}|x_0-y||a(y)|\,dy\r)^{2}t\,dt\\
&\leq \sum_{j=1}^{n+1}\int_{|x-x_0|/4}^{\infty}
t^{-2(n+2)}r^{2}\lf(\int_B|a(y)|\,dy\r)^{2}t\,dt\\
&\ls (n+1)r^{2}r^{2n}\|a\|_{L^{q}_{B}}^{2}\int_{|x-x_0|/4}^{\infty}t^{-(2n+3)}\,dt\\
&\ls (n+1)\lf(\frac{r}{|x-x_0|}\r)^{2(n+1)}\|a\|_{L^{q}_{B}}^{2}.
\end{align*}
Combining the estimates of ${\rm J}_{1}$ and ${\rm J}_{2}$,
we conclude that for any $x\in(2^{k+1}B)\setminus (2^{k}B)$,
\begin{align*}
[\mathcal{G}(a)(x)]^{2}
\ls {\rm J}_{1}+{\rm J}_{2}
\ls \lf(\frac{r}{|x-x_0|}\r)^{2(n+1)}\|a\|_{L^{q}_{B}}^{2}
\ls 2^{-2k(n+1)}\|a\|_{L^{q}_{B}}^{2},
\end{align*}
which implies
$$\mathcal{G}(a)(x)\ls 2^{-k(n+1)}\|a\|_{L^{q}_{B}}.$$
Then
\begin{align*}
&\lf\|(b-b_B)\mathbf{1}_{(2B)^c}\mathcal{G}(a)\r\|_{L^{1}(\rn)}\\
&\quad\leq \sum_{k=1}^{\infty}\int_{(2^{k+1}B)\setminus(2^{k}B)}|b(x)-b_B|\lf|\mathcal{G}(a)(x)\r|\,dx\\
&\quad\ls \|a\|_{L^{q}_{B}}\sum_{k=1}^{\infty}2^{-k(n+1)}
\lf[\int_{2^{k+1}B}|b(x)-b_{2^{k+1}B}|
\,dx+|2^{k+1}B||b_B-b_{2^{k+1}B}|\r].
\end{align*}
By repeating the proof of \eqref{eq.2.18}, we immediately obtain
\begin{align*}
\lf\|(b-b_B)\mathbf{1}_{(2B)^c}\mathcal{G}(a)\r\|_{L^{1}(\rn)}
\ls \|a\|_{L^{q}_{B}}\|{\mathbf 1}_{B}\|_{L^{\varphi}(\mathbb{R}^n)}
\|b\|_{\mathrm{BMO}_{\varphi}(\rn)}.
\end{align*}
So, \eqref{G-b-bB-A} holds. We finish the proof of Proposition \ref{proposition-g-func}.
\end{proof}

\begin{proposition}\label{proposition-mar}
Let $\varphi$ be a growth function satisfying $n/(n+1)<i(\varphi)\leq I(\varphi)<1$.
Suppose that $\Omega\in\mathrm{Din}^{p}_{\beta}$ with $p\in(1,\,\infty]$ and $\beta\in(0,1)$. Set
$$\epsilon=\min\lf\{\beta, \ \frac{1}{2}\left(1-\frac{1}{p}\right)\r\}.$$
If the critical lower type index $i(\varphi)>n/(n+\epsilon)$, then
Marcinkiewicz integral $\mu_\Omega\in \mathcal{K}_{\varphi}$.
\end{proposition}
\begin{proof}
Three fundamental boundedness properties of Marcinkiewicz integral $\mu_\Omega$ have been well documented in the existing literature: the $H^{1}(\rn)$ to $L^{1}(\rn)$ boundedness is derived from \cite[Theorem~1]{Dingyong-2002}, the boundedness from $L^{1}(\rn)$ to $L^{1,\,\infty}(\rn)$ is proved in Benedek \cite{Benedek-1962}, and the full range $L^{q}(\rn)$ boundedness for all $q\in(1,\infty)$ is established in the pioneering work Benedek \cite{Benedek-1962}. As a consequence, to show that Marcinkiewicz integral $\mu_\Omega$ belongs to the class defined in Definition \ref{def2.1}, we only need to verify the third condition (iii) therein. Similar to the arguments that $\mathbf{S}$ and $\mathcal{G}$ satisfy Definition \ref{def2.1}(iii), by
Lemma \ref{lemma2.5}, we have
\begin{align*}
\|(b-b_{B})\mu_\Omega(a)\|_{L^{1}(\mathbb{R}^{n})}
&\ls \|a\|_{L^{q}_{B}}\|{\mathbf{1}_{B}\|_{L^{\varphi}(\rn)}
\|b\|_{\mathrm{BMO}_{\varphi}(\rn)}}
+\|(b-b_B)\mathbf{1}_{(2B)^c} \mu_\Omega(a)\|_{L^{1}(\rn)}, 
\end{align*}
where $a$ is a $(\varphi,q,0)$-atom with $\operatorname{supp}a \subset B=B(x_0, r)$ and $q\in(1,\infty)$.
Hence, it remains to prove that
\begin{align}\label{b-bbOmega}
\|(b-b_B)\mathbf{1}_{(2B)^c}\mu_\Omega(a)\|_{L^{1}(\rn)}\ls \|a\|_{L^{q}_{B}}\|{\mathbf{1}_{B}\|_{L^{\varphi}(\rn)}
\|b\|_{\mathrm{BMO}_{\varphi}(\rn)}}.
\end{align}

In order to prove \eqref{b-bbOmega}, by the condition $i(\varphi)>n/(n+\epsilon)$ with $\epsilon=\min\{\beta,\,\frac{1}{2}(1-1/p)\}$, we choose $\sigma$ such that $\varphi$ is of lower type $\sigma$ satisfying $n/(n+\epsilon)<\sigma<i(\varphi)$. So
\begin{align}\label{eq.2.20}
n\left(1/\sigma-1\right)<\epsilon.
\end{align}
Write
\begin{align*}
&\|(b-b_B)\mathbf{1}_{(2B)^c} \mu_\Omega(a)\|_{L^{1}(\rn)}\\
&\quad\leq\int_{(2B)^c}|b(x)-b_B|\left(\int_0^{|x-x_0|+r}\left|\int_{|x-y|\leq t}\frac{\Omega(x-y)}{|x-y|^{n-1}}a(y)\,dy\right|^{2}\frac{dt}{t^3}\right)^{1/2}\,dx\\
&\qquad+\int_{(2B)^c}|b(x)-b_B|\left(\int_{|x-x_0|+r}^{\infty}\left|\int_{|x-y|\leq t}\frac{\Omega(x-y)}{|x-y|^{n-1}}a(y)\,dy\right|^{2}\frac{dt}{t^3}\right)^{1/2}\,dx\\
&\quad=:{\rm U_{1}}+{\rm U_{2}}.
\end{align*}

First, we estimate ${\rm U_{1}}$. From the arguments of \cite[p.174]{Hyy-Whx-2023}, it follows that
\begin{align*}
{\rm U_{1}}
&=\int_{(2B)^c}|b(x)-b_B|\left(\int_0^{|x-x_0|+r}\left|\int_{|x-y|\leq t}\frac{\Omega(x-y)}{|x-y|^{n-1}}a(y)\,dy\right|^{2}\frac{dt}{t^3}\right)^{1/2}\,dx\\
&\ls r^{\frac{1}{2}}\int_{B}\int_{(2B)^c}\frac{\Omega(x-y)}{|x-y|^{n+1/2}}|b(x)-b_B|\,dx
|a(y)|\,dy\\
&\ls r^{\frac{1}{2}}\int_{B}\lf[\lf(\int_{(2B)^c}\frac{|\Omega(x-y)|^{p}}{|x-y|^{n+1/2}}\,dx\r)^{1/p}
\lf(\int_{(2B)^c}\frac{|b(x)-b_B|^{p'}}{|x-y|^{n+1/2}}\,dx\r)^{1/p'}
\r]|a(y)|\,dy\\
&\ls r^{\frac{1}{2}-\frac{1}{2p}}\left(\int_{(2B)^{c}}
\frac{|b(x)-b_B|^{p'}}{|x-y|^{n+1/2}}\,dx\right)^{1/p'}\int_B |a(y)|\,dy,
\end{align*}
where the last step from \cite[p.174]{Hyy-Whx-2023}, that is,
$$\int_{(2B)^c}\frac{|\Omega(x-y)|^{p}}{|x-y|^{n+1/2}}\,dx\ls r^{-\frac{1}{2}}.$$
Further, for $y\in B=B(x_0, r)$ and $x\in(2B)^{c}$, we deduce that $|x-x_0|\geq 2r$ and $|y-x_0|\leq r$, which implies
$$|x-y|\geq|x-x_0|-|y-x_0|\geq|x-x_0|-r\geq\frac{|x-x_0|}{2}$$
and
$$|x-y|\leq|x-x_0|+|y-x_0|\leq|x-x_0|+r\leq\frac{3|x-x_0|}{2}.$$
Then
\begin{align*}
|x-y|\approx|x-x_0|.
\end{align*}
By this, \eqref{Ba-BAQ}, \eqref{eq.B-1}, \eqref{eq.B-2} and \eqref{eq.simga-k1}, we see that
\begin{align*}
{\rm U_{1}}
&\ls r^{\frac{1}{2}-\frac{1}{2p}}\left(\int_{(2B)^{c}}
\frac{|b(x)-b_B|^{p'}}{|x-x_0|^{n+1/2}}\,dx\right)^{1/p'}\int_B |a(y)|\,dy\\
&\ls r^{\frac{1}{2}-\frac{1}{2p}}
r^{n}\|a\|_{L^{q}_{B}}\sum_{k=1}^{\infty}\left(\int_{(2^{k+1}B)\setminus(2^{k}B)}
\frac{|b(x)-b_B|^{p'}}{|x-x_0|^{n+1/2}}\,dx\right)^{1/p'}\\
&\leq r^{\frac{1}{2}-\frac{1}{2p}}
r^{n}r^{-(n+1/2)\frac{1}{p'}}\|a\|_{L^{q}_{B}}
\sum_{k=1}^{\infty}2^{-k(n+1/2)\frac{1}{p'}}\times\Bigg[\left(\int_{2^{k+1}B}
|b(x)-b_{2^{k+1}B}|^{p'}\,dx\right)^{1/p'}\\
&\qquad+|2^{k+1}B|^{1/p'}
|b_B-b_{2^{k+1}B}|\Bigg]\\
&\ls r^{\frac{n}{p}}\|a\|_{L^{q}_{B}}
\sum_{k=1}^{\infty}2^{-k(n+1/2)\frac{1}{p'}}\times\Bigg[|2^{k+1}B|^{1/p'-1}\|{\mathbf 1}_{2^{k+1}B}\|_{L^{\varphi}(\mathbb{R}^n)}\|b\|_{\mathrm{BMO}_{\varphi}(\rn)} \\
&\qquad + |2^{k+1}B|^{1/p'}|2^{k+1}B|^{-1}2^{\frac{(k+1)n}{\sigma}}\|{\mathbf 1}_{B}\|_{L^{\varphi}(\mathbb{R}^n)}\|b\|_{\mathrm{BMO}_{\varphi}(\rn)}\Bigg]\\
&\ls\|a\|_{L^{q}_{B}}\|{\mathbf 1}_{B}\|_{L^{\varphi}(\mathbb{R}^n)}\|b\|_{\mathrm{BMO}_{\varphi}(\rn)}
\sum_{k=1}^{\infty}2^{-k(n+1/2)\frac{1}{p'}-(k+1)\frac{n}{p}+\frac{(k+1)n}{\sigma}}\\
&\approx\|a\|_{L^{q}_{B}}\|{\mathbf 1}_{B}\|_{L^{\varphi}(\mathbb{R}^n)}\|b\|_{\mathrm{BMO}_{\varphi}(\rn)}
\sum_{k=1}^{\infty}2^{-k(n+\frac{1}{2}(1-1/p)-n/\sigma)}.
\end{align*}

Now, we begin estimate ${\rm U_{2}}$. By the discussion of \cite[p.175]{Hyy-Whx-2023}, we have
\begin{align*}
{\rm U_{2}}
&=\int_{(2B)^c}|b(x)-b_B|\left(\int_{|x-x_0|+r}^{\infty}\left|\int_{|x-y|\leq t}\frac{\Omega(x-y)}{|x-y|^{n-1}}a(y)\,dy\right|^{2}\frac{dt}{t^3}\right)^{1/2}\,dx\\
&\ls \int_B |a(y)|
\sum_{k=1}^{\infty}\int_{(2^{k+1}B)\setminus(2^{k}B)}\left|\frac{\Omega(x-y)}{|x-y|^{n-1}}-\frac{\Omega(x-x_0)}{|x-x_0|^{n-1}}
\right|\times\frac{|b(x)-b_B|}{2^{k}r}\,dx\,dy\\
&\ls r^{-1}\int_B |a(y)|
\sum_{k=1}^{\infty}2^{-k}\left(\int_{(2^{k+1}B)\setminus(2^{k}B)}
\left|\frac{\Omega(x-y)}{|x-y|^{n-1}}
-\frac{\Omega(x-x_0)}{|x-x_0|^{n-1}}\right|^{p}\,dx\right)^{1/p}\,dy\\
&\qquad \times \left(\int_{(2^{k+1}B)\setminus(2^{k}B)}|b(x)-b_B|^{p'}\,dx\right)^{1/p'}.
\end{align*}
Note that we recall \cite[Lemma~2.4]{Hyy-Whx-2023} and obtain
\begin{align*}
&\left(\int_{(2^{k+1}B)\setminus(2^{k}B)}
\left|\frac{\Omega(x-y)}{|x-y|^{n-1}}-\frac{\Omega(x-x_0)}{|x-x_0|^{n-1}}\right|^{p}\,dx\right)^{1/p}\\
&\quad \ls (2^{k}r)^{n/p-n+1}\lf(\frac{|y-x_0|}{2^{k}r}+\int_{|y-x_0|/(2^{k+1}r)}^{|y-x_0|/(2^{k}r)}
\frac{\omega_{p, \, \delta}(\Omega)}{\delta}\,d\delta\r).
\end{align*}
Therefore, by the estimate above, \eqref{Ba-BAQ}, \eqref{eq.B-1}, \eqref{eq.B-2} and \eqref{eq.simga-k1}, it follows that
\begin{align*}
{\rm U_{2}}
&\ls r^{-1}\int_B |a(y)|
\sum_{k=1}^{\infty}2^{-k}(2^{k}r)^{n/p-n+1}\lf(\frac{|y-x_0|}{2^{k}r}
+\int_{|y-x_0|/(2^{k+1}r)}^{|y-x_0|/(2^{k}r)}
\frac{\omega_{p, \, \delta}(\Omega)}{\delta}\,d\delta\r)\,dy\\
&\qquad \times \left(\int_{(2^{k+1}B)\setminus(2^{k}B)}|b(x)-b_B|^{p'}\,dx\right)^{1/p'}\\
&\ls r^{-1}\int_B |a(y)|\,dy
\sum_{k=1}^{\infty}2^{-k}(2^{k}r)^{n/p-n+1}\lf(2^{-k}+2^{-k\beta}
\int_{|y-x_0|/(2^{k+1}r)}^{|y-x_0|/(2^{k}r)}
\frac{\omega_{p, \, \delta}(\Omega)}{\delta^{1+\beta}}\,d\delta\r)\\
&\qquad \times \lf[\left(\int_{2^{k+1}B}|b(x)-b_{2^{k+1}B}|^{p'}\,dx\right)^{1/p'}
+|2^{k+1}B|^{1/p'}|b_{B}-b_{2^{k+1}B}|\r]\\
&\ls r^{-1}r^{n}\|a\|_{L^{q}_{B}}
\sum_{k=1}^{\infty}2^{-k}(2^{k}r)^{n/p-n+1}2^{-k\beta}
\left(1+\int_{0}^{1}\frac{\omega_{p, \, \delta}(\Omega)}{\delta^{1+\beta}}\,d\delta\right)\\
&\qquad \times
\left[|2^{k+1}B|^{1/p'-1}2^{\frac{(k+1)n}{\sigma}}\|{\mathbf 1}_{B}\|_{L^{\varphi}(\mathbb{R}^n)}\|b\|_{\mathrm{BMO}_{\varphi}(\rn)}\right]\\
&\approx\|a\|_{L^{q}_{B}}\|{\mathbf 1}_{B}\|_{L^{\varphi}(\mathbb{R}^n)}
\|b\|_{\mathrm{BMO}_{\varphi}(\rn)}\left(1+\int_{0}^{1}\frac{\omega_{p, \, \delta}(\Omega)}{\delta^{1+\beta}}\,d\delta\right)
\sum_{k=1}^{\infty}2^{-k(n+\beta-n/\sigma)},
\end{align*}
where the second step used $|y-x_0|\leq r$ and
$$\frac{2^{-k}}{\delta}>\frac{2^{-k}2^{k}r}{|y-x_0|}>1.$$
Further, we apply $L^p$-Dini condition of order $\beta$ to derive
\begin{align*}
{\rm U_{2}}
\ls\|a\|_{L^{q}_{B}}\|{\mathbf 1}_{B}\|_{L^{\varphi}(\mathbb{R}^n)}
\|b\|_{\mathrm{BMO}_{\varphi}(\rn)}
\sum_{k=1}^{\infty}2^{-k(n+\beta-n/\sigma)}.
\end{align*}

Finally, combining the estimates of ${\rm U_{1}}$ and ${\rm U_{2}}$, we derive
that
\begin{align*}
{\rm U_{1}}+{\rm U_{2}}
&\ls \|a\|_{L^{q}_{B}}\|{\mathbf 1}_{B}\|_{L^{\varphi}(\mathbb{R}^n)}\|b\|_{\mathrm{BMO}_{\varphi}(\rn)}
\left(\sum_{k=1}^{\infty}2^{-k(n+\frac{1}{2}(1-1/p)-n/\sigma)}
+\sum_{k=1}^{\infty}2^{-k(n+\beta-n/\sigma)}\right)\\
&\ls \|a\|_{L^{q}_{B}}\|{\mathbf 1}_{B}\|_{L^{\varphi}(\mathbb{R}^n)}\|b\|_{\mathrm{BMO}_{\varphi}(\rn)}\sum_{k=1}^{\infty}2^{-k(n+\epsilon-n/\sigma)}.
\end{align*}
According to $n+\epsilon-n/\sigma=\epsilon-n(1/\sigma-1)$ and \eqref{eq.2.20}, we conclude that
\begin{align*}
\|(b-b_B)\mathbf{1}_{(2B)^c} \mu_\Omega(a)\|_{L^{1}(\rn)}
\leq {\rm U_{1}}+{\rm U_{2}}
\ls \|a\|_{L^{q}_{B}}\|{\mathbf 1}_{B}\|_{L^{\varphi}(\mathbb{R}^n)}\|b\|_{\mathrm{BMO}_{\varphi}(\rn)}.
\end{align*}
Then \eqref{b-bbOmega} holds. So, $\mu_\Omega\in \mathcal{K}_{\varphi}$.
\end{proof}

\begin{proposition}\label{proposition-BR}
Let $0<R<\infty$ and $\varphi$ be a growth function.
For $\delta>(n-1)/2$ and some $\delta_0$ satisfying
$$(n-1)/2<\delta_0<\min\{\delta,\,(n+1)/2\},$$
if the critical lower type index $i(\varphi)\in(2n/(n+1+2\delta_0), 1]$, then Bochner-Riesz mean operator $T_R^{\delta}\in\mathcal{K}_{\varphi}$ and Bochner-Riesz maximal operator $T_{*}^{\delta}\in \mathcal{K}_{\varphi}$.
\end{proposition}
\begin{proof}
We first collect three known boundedness properties of $T^{\delta}_R$ and $T_{*}^{\delta}$ from the existing literature: From \cite{P.-1981, Tai-Weiss-1980}, it follows that $T^{\delta}_R$ and $T_{*}^{\delta}$ map $H^{1}(\rn)$ boundedly into $L^{1}(\rn)$;
by \cite[p.389]{Stein-Ana-2}, they are bounded from $L^{1}(\rn)$ to $L^{1,\,\infty}(\rn)$  and they admit bounded action on every $L^{q}(\rn)$ with $q\in(1,\infty)$. These three properties already cover conditions (i) and (ii) in Definition \ref{def2.1}, so the only task left for the subsequent argument is to verify the third condition (iii).
By following a discussion similar to $\mathbf{S}$, $\mathcal{G}$ and $\mu_\Omega$,
we use Lemma \ref{lemma2.5} to obtain
\begin{align*}
\|(b-b_{B})T_R^{\delta}(a)\|_{L^{1}(\mathbb{R}^{n})}
&\ls \|a\|_{L^{q}_{B}}\|{\mathbf{1}_{B}\|_{L^{\varphi}(\rn)}
\|b\|_{\mathrm{BMO}_{\varphi}(\rn)}}
+\|(b-b_B)\mathbf{1}_{(2B)^c}T_R^{\delta}(a)\|_{L^{1}(\rn)}
\end{align*}
and
\begin{align*}
\|(b-b_{B})T_{*}^{\delta}(a)\|_{L^{1}(\mathbb{R}^{n})}
&\ls \|a\|_{L^{q}_{B}}\|{\mathbf{1}_{B}\|_{L^{\varphi}(\rn)}
\|b\|_{\mathrm{BMO}_{\varphi}(\rn)}}
+\|(b-b_B)\mathbf{1}_{(2B)^c}T_{*}^{\delta}(a)\|_{L^{1}(\rn)},
\end{align*}
where $b\in\mathrm{BMO}_{\varphi}(\rn)$ and $a$ is a $(\varphi,q,0)$-atom with $\operatorname{supp}a \subset B$ and $q\in(1,\infty)$.
So, it suffices to prove that
\begin{align}\label{eq-b-bB-bR-t1}
\|(b-b_B)\mathbf{1}_{(2B)^c}T_R^{\delta}(a)\|_{L^{1}(\rn)}\ls \|a\|_{L^{q}_{B}}\|{\mathbf{1}_{B}\|_{L^{\varphi}(\rn)}\|b\|_{\mathrm{BMO}_{\varphi}(\rn)}}
\end{align}
and
\begin{align}\label{eq-b-bB-bR-t2}
\|(b-b_B)\mathbf{1}_{(2B)^c}T_{*}^{\delta}(a)\|_{L^{1}(\rn)}\ls \|a\|_{L^{q}_{B}}\|{\mathbf{1}_{B}\|_{L^{\varphi}(\rn)}\|b\|_{\mathrm{BMO}_{\varphi}(\rn)}}.
\end{align}

Now, we verify \eqref{eq-b-bB-bR-t1} and \eqref{eq-b-bB-bR-t2}. Fix $B=B(x_0, r)$. By the condition $i(\varphi)>2n/(n+1+2\delta_0)$, we pick $\sigma$ such that $\varphi$ is of lower type $\sigma$ satisfying
\begin{align}\label{eq.2.31}
2n/(n+1+2\delta_0)<\sigma<i(\varphi).
\end{align}
If we claim
\begin{align}\label{eq.2.32}
|T_R^{\delta}(a)(x)|\ls 2^{-k(\delta_0+\frac{n+1}{2})}\|a\|_{L^{q}_{B}},
\end{align}
then, by \eqref{eq.B-1}, \eqref{eq.B-2}, \eqref{eq.simga-k1} and \eqref{eq.2.32}, we see that
\begin{align}\label{eq-bbB-TR13}
&\|(b-b_B)\mathbf{1}_{(2B)^c}T_R^{\delta}(a)\|_{L^{1}(\rn)}\\
&\quad\leq \sum_{k=1}^{\infty}\int_{(2^{k+1}B)\setminus(2^{k}B)}|b(x)-b_B||T_R^{\delta}(a)(x)|\,dx\notag\\
&\quad\ls \|a\|_{L^{q}_{B}}\sum_{k=1}^{\infty}2^{-k(\delta_0+\frac{n+1}{2})}
\lf(\int_{2^{k+1}B}|b(x)-b_{2^{k+1}B}|\,dx
+|2^{k+1}B||b_B-b_{2^{k+1}B}|\r)\notag\\
&\quad\ls \|a\|_{L^{q}_{B}}\sum_{k=1}^{\infty}2^{-k(\delta_0+\frac{n+1}{2})}\left(\|{\mathbf 1}_{2^{k+1}B}\|_{L^{\varphi}(\mathbb{R}^n)}
\|b\|_{\mathrm{BMO}_{\varphi}(\rn)}+
2^{\frac{(k+1)n}{\sigma}}\|b\|_{\mathrm{BMO}_{\varphi}(\rn)}\|{\mathbf 1}_{B}\|_{L^{\varphi}(\mathbb{R}^n)}\right)\notag\\
&\quad\ls \|a\|_{L^{q}_{B}}\|{\mathbf 1}_{B}\|_{L^{\varphi}(\mathbb{R}^n)}\|b\|_{\mathrm{BMO}_{\varphi}(\rn)} \sum_{k=1}^{\infty}2^{-k\left(\delta_0+\frac{n+1}{2}-\frac{n}{\sigma}\right)}\notag\\
&\quad\ls \|a\|_{L^{q}_{B}}\|{\mathbf 1}_{B}\|_{L^{\varphi}(\mathbb{R}^n)}\|b\|_{\mathrm{BMO}_{\varphi}(\rn)}\notag,
\end{align}
where the last step derives from $\delta_0+\frac{n+1}{2}-\frac{n}{\sigma}>0$
since the assumption \eqref{eq.2.31}. Thus, \eqref{eq-b-bB-bR-t1} holds.
Moreover, by using \eqref{eq.2.32} and repeating the argument of \eqref{eq-bbB-TR13},
it follows that
\begin{align*}
|T_{*}^{\delta}(a)(x)|\ls 2^{-k(\delta_0+\frac{n+1}{2})}\|a\|_{L^{q}_{B}}
\end{align*}
and \eqref{eq-b-bB-bR-t2} holds.

It remains to show \eqref{eq.2.32}.
For ball $B=B(x_0, r)$, we divide the two cases $1/R>r$ and $0<1/R<r$ to discuss \eqref{eq.2.32}.

\medskip

{\it Case 1: $1/R>r$.}
For $y\in B=B(x_0, r)$ and $x\in(2^{k+1}B)\setminus(2^{k}B)$, there exists $\theta\in (0,1)$ and a point $\xi_y=\theta(x-y)+(1-\theta)(x-x_0)$ lying on the line segment joining $x-y$ and $x-x_0$, such that
\begin{align*}
|T_R^{\delta}(a)(x)|
&=\left|\int_B a(y)\phi_{1/R}(x-y)\,dy\right|\\
&=\left|\int_B a(y)(\phi_{1/R}(x-y)-\phi_{1/R}(x-x_0))\,dy\right|\\
&\leq \int_B |a(y)||(\phi_{1/R}(x-y)-\phi_{1/R}(x-x_0))|\,dy\\
&=\int_B |a(y)||y-x_0||\nabla\phi_{1/R}(\xi_y)|\,dy.
\end{align*}
From the definition $\phi_{1/R}(\xi_y)=R^{n}\phi(R\xi_y)$, it follows that
\begin{align*}
\nabla\phi_{1/R}(\xi_y)&=\nabla(R^{n}\phi(R\xi_y))
=R^{n}\cdot\nabla[\phi(R\xi_y)]
=R^{n+1}(\nabla\phi)(R\xi_y).
\end{align*}
Then, by \eqref{Sato-2}, \eqref{Ba-BAQ}, $1/R>r$ and
\begin{align*}
|\xi_y|&=|\theta(x-y)+(1-\theta)(x-x_0)|\\
&=|(x-x_0)-\theta(y-x_0)|
\geq |x-x_0|-|y-x_0|
\geq \frac{|x-x_0|}{2}
\geq \frac{2^{k}r}{2}
\quad \text{with} \ y\in B(x_0, r),
\end{align*}
we have
\begin{align*}
|T_R^{\delta}(a)(x)|
&\ls rR^{n+1}(1+R|\xi_y|)^{-(\delta+\frac{n+1}{2})}\int_B |a(y)|\,dy\\
&\ls rR^{n+1}\left(1+R\frac{2^{k}r}{2}\right)^{-\left(\delta_0+\frac{n+1}{2}\right)}
r^{n}\|a\|_{L^{q}_{B}}\\
&\ls 2^{-k\left(\delta_0+\frac{n+1}{2}\right)}r^{\frac{n+1}{2}-\delta_0}
R^{\frac{n+1}{2}-\delta_0}\|a\|_{L^{q}_{B}}\\
&\ls 2^{-k\left(\delta_0+\frac{n+1}{2}\right)}r^{\frac{n+1}{2}-\delta_0}(1/r)^{\frac{n+1}{2}-\delta_0}
\|a\|_{L^{q}_{B}}\\
&\approx 2^{-k\left(\delta_0+\frac{n+1}{2}\right)}\|a\|_{L^{q}_{B}}.
\end{align*}
This implies \eqref{eq.2.32} in the case $1/R>r$.

\medskip

{\it Case 2: $0<1/R\leq r$.}
Due to $y\in B=B(x_0, r)$ and $x\in(2^{k+1}B)\setminus(2^{k}B)$,
it follows from \eqref{Sato-2} that
$$|x-y|\geq|x-x_0|-|y-x_0|\geq |x-x_0|-r\geq\frac{|x-x_0|}{2}\geq \frac{2^{k}r}{2} ,$$
by this, \eqref{Sato-2} and \eqref{Ba-BAQ}, we immediately obtain that
\begin{align*}
|T_R^{\delta}(a)(x)|
&\leq \int_B |a(y)||\phi_{1/R}(x-y)|\,dy\\
&\ls R^{n}\int_B |a(y)|\frac{1}{(1+R|x-y|)^{\delta+\frac{n+1}{2}}}\,dy\\
&\ls R^{n}\int_B |a(y)|\frac{1}{(1+\frac{R}{2}2^{k}r)^{\delta_0+\frac{n+1}{2}}}\,dy\\
&\ls R^{n-\left(\delta_0+\frac{n+1}{2}\right)}\lf(\frac{R}{2}2^{k}r\r)^{-\delta_{0}-\frac{n+1}{2}}
\int_B |a(y)|\,dy\\
&\ls 2^{-k\left(\delta_0+\frac{n+1}{2}\right)}
R^{\frac{n-1}{2}-\delta_0}r^{-\left(\delta_0+\frac{n+1}{2}\right)}r^{n}\|a\|_{L^{q}_{B}}\\
&\approx 2^{-k\left(\delta_0+\frac{n+1}{2}\right)}(Rr)^{\frac{n-1}{2}-\delta_0}\|a\|_{L^{q}_{B}}\\
&\ls 2^{-k\left(\delta_0+\frac{n+1}{2}\right)}\|a\|_{L^{q}_{B}},
\end{align*}
where the last step used $0<1/R\leq r$ and $\frac{n-1}{2}-\delta_0<0$.
This implies \eqref{eq.2.32} in the case $0<1/R\leq r$.
We finish the proof of Proposition \ref{proposition-BR}.
\end{proof}

\section{Proofs of Theorems \ref{main-Thm-1}, \ref{Theo-main-b-2}
and \ref{Theo-main-b-3}${\rm(i)}$-${\rm(ii)}$}\label{Sec-3}
The main purpose of this section is to prove Theorems \ref{Theo-main-c-1} and \ref{Theo-main-c-2} by using \cite[Lemma~5.8 and Theorem~5.10]{F-L2024},
which implies Theorems \ref{main-Thm-1}, \ref{Theo-main-b-2} and \ref{Theo-main-b-3}(i)-(ii).

\begin{Theorem}\label{Theo-main-c-1}
Let $\varphi$ be a growth function satisfying $0<i(\varphi)\leq I(\varphi)<1$ and \eqref{eq.1.5}.
Denote $\widetilde{\mathcal{T}}=\{\mathbf{S},\,\mathcal{G},\,\mu_\Omega,\,T_R^{\delta},\, T_*^{\delta}\}$, with each operator satisfying the corresponding hypotheses, 
detailed in Theorems \ref{main-Thm-1}, \ref{Theo-main-b-2} and \ref{Theo-main-b-3}.
If $b\in\mathrm{BMO}_{\varphi}(\rn)$ and $\mathfrak{T}\in\widetilde{\mathcal{T}}$,
then, there exists a constant $C>0$ such that for any $h\in H^{\varphi}_{b}(\rn)$,
$$\|[b, \mathfrak{T}]h\|_{L^{1}(\rn)}\leq C \|h\|_{H_b^{\varphi}(\rn)}.$$
\end{Theorem}
\begin{proof}
It follows from \cite[Lemma~5.8 and Theorem~5.10]{F-L2024} that there exists an operator
 \begin{align}\label{hua-R}
\mathfrak{R}: H^{\varphi}(\mathbb{R}^{n})\times\mathrm{BMO}_{\varphi}(\rn)\to L^{1}(\rn)
 \end{align}
 such that
 \begin{align}\label{bhT-pan4}
|[b, \mathfrak{T}]h|\leq |\mathfrak{T}(\Pi_{4}(h,\, b))|+\mathfrak{R}(h,\, b),
 \end{align}
where $b\in\mathrm{BMO}_{\varphi}(\rn)$, $h\in H^{\varphi}_{b}(\rn)$, $\mathfrak{T}\in\widetilde{\mathcal{T}}$ and the term $\Pi_{4}(h,\, b)$, which belongs to $H^{1}(\mathbb{R}^{n})$, is defined as in \cite[(3.3)]{F-L2024} by
\begin{align}\label{pan4-hb}
\Pi_{4}(h,\, b)=\dsum_{I\in\mathcal{D}}\sum_{\lambda\in E}
\langle h,\, \psi^{\lambda}_{I}\rangle\langle b,\, \psi^{\lambda}_{I}\rangle
(\psi^{\lambda}_{I})^{2}.
 \end{align}
For more detailed explanations of the concept $\Pi_{4}(h,\, b)$, we refer readers to \cite[Section~3.1]{F-L2024}, and no redundant elaboration will be presented here.
Then, by \eqref{bhT-pan4}, Propositions
\ref{proposition-Larea-S}-\ref{proposition-g-func}-\ref{proposition-mar}-\ref{proposition-BR} and \cite[Theorem~5.10]{F-L2024}, we see that
\begin{align*}
 \|[b, \mathfrak{T}]h\|_{L^{1}(\rn)}&\leq \|\mathfrak{T}(\Pi_{4}(h,\, b))\|_{L^{1}(\rn)}+\|\mathfrak{R}(h,\, b)\|_{L^{1}(\rn)}\\
 &\ls\|\Pi_{4}(h,\, b)\|_{H^{1}(\rn)}+\|h\|_{H^{\varphi}(\mathbb{R}^{n})}\|b\|_{\mathrm{BMO}_{\varphi}(\rn)}\\
 &\approx \|h\|_{H^{\varphi}_{b}(\mathbb{R}^{n})},
 \end{align*}
 as desired.
\end{proof}

\begin{Theorem}\label{Theo-main-c-2}
Let $\varphi$ and the family of operators $\widetilde{\mathcal{T}}$ be as in Theorem \ref{Theo-main-c-1}. If $b\in\mathrm{BMO}_{\varphi}(\rn)$ and $\mathfrak{T}\in\widetilde{\mathcal{T}}=\{\mathbf{S},\,\mathcal{G},\,\mu_\Omega,\,T_R^{\delta},\, T_*^{\delta}\}$,
then there exists a constant $C>0$ such that for any $h\in H^{\varphi}(\rn)$,
$$\|[b, \mathfrak{T}]h\|_{L^{1,\,\infty}(\rn)}\leq C \|h\|_{H^{\varphi}(\rn)}.$$
\end{Theorem}
\begin{proof}
From $\mathfrak{T}\in\widetilde{\mathcal{T}}$ and Propositions
\ref{proposition-Larea-S}-\ref{proposition-g-func}-\ref{proposition-mar}-\ref{proposition-BR},
it follows that $\mathfrak{T}$ is bounded from $L^{1}(\rn)$ to $L^{1,\,\infty}(\rn)$. By this, \eqref{bhT-pan4}, $L^{1}(\rn)\subset L^{1,\,\infty}(\rn)$ and the boundedness properties of $\mathfrak{R}$ and $\Pi_{4}$ from $H^{\varphi}(\mathbb{R}^{n})\times\mathrm{BMO}_{\varphi}(\rn)$
to $L^{1}(\rn)$ (see \cite[Proposition~3.17 and Lemma~5.8]{F-L2024}), we conclude that
 \begin{align*}
 \|[b, \mathfrak{T}]h\|_{L^{1,\,\infty}(\rn)}&\ls \|\mathfrak{T}(\Pi_{4}(h,\, b))\|_{L^{1,\,\infty}(\rn)}+\|\mathfrak{R}(h,\, b)\|_{L^{1,\,\infty}(\rn)}\\
 &\ls \|\Pi_{4}(h,\, b)\|_{L^{1}(\rn)}+\|\mathfrak{R}(h,\, b)\|_{L^{1}(\rn)}\\
 &\ls \|h\|_{H^{\varphi}(\mathbb{R}^{n})}\|b\|_{\mathrm{BMO}_{\varphi}(\rn)},
 \end{align*}
 where $\mathfrak{R}$ and $\Pi_{4}$ adopt the definitions given in \eqref{hua-R} and \eqref{pan4-hb}, respectively.
\end{proof}

\section{Proof of Theorem \ref{Theo-main-b-3}{\rm(iii)} and counterexample}
In this section we introduce the definition of $H^{1}(\rn)$-molecular from \cite[p. 71]{Tai-Weiss-1980} firstly, then we complete the proof of the boundedness of the commutator
$$[b, T_R^{\delta}]: \ H_b^{\varphi}(\rn)\longrightarrow H^1(\rn),$$
where $b\in\mathrm{BMO}_{\varphi}(\rn)$ and $T_R^{\delta}$ is the Bochner Riesz mean operator. This implies Theorem \ref{Theo-main-b-3}(iii). 
It is worth noting that the commutators associated with the
Lusin area integral, g-function, Marcinkiewicz integral and Bochner Riesz maximal operator fail to satisfy the aforementioned boundedness.
To explain this observation, we give a counterexample in Theorem \ref{the-ex-4.1},
which relies essentially on Lemmas \ref{lemma4.3} and \ref{lemma4.4}.

\begin{definition}\label{def.4.1}
Let $p\in(1, \infty]$ and $\vec{\nu}=\{\nu_{k}\}_{k\in\nn}$ be a nonnegative sequence satisfying $\sum_{k=1}^{\infty}k\nu_{k}<\infty$.
Given a ball $B\subset\rn$, a measurable function $\omega$ is called a $(1,\, p,\, \vec{\nu})$-molecule relative to $B$ if:
\begin{enumerate}
\item[\rm (i)] $\|\omega{\mathbf 1}_{B}\|_{L^p(\rn)}\leq|B|^{1/p-1}$ on $\rn$;
\item[\rm (ii)] $\|\omega{\mathbf 1}_{(2^kB)\setminus
(2^{k-1}B)}\|_{L^p(\rn)}\leq\nu_{k}|2^{k}B|^{1/p-1}$  for any $k\in\nn$;
\item[\rm (iii)] $\int_{\rn}\omega(x)\,dx=0$.
\end{enumerate}
\end{definition}

%\begin{theorem}\label{Theo-main-c-3}
%Given a growth function $\varphi$ satisfying the conditions of Theorem \ref{Theo-main-b-3}, and let $\delta, \delta_0$ be as in Theorem \ref{Theo-main-b-3}. For any $b\in\mathrm{BMO}_{\varphi}(\rn)$, if $T_R^{\delta}$ satisfies \eqref{eq.2.36},
%then there exists a constant $C>0$ such that for any $h\in H^{\varphi}_{b}(\rn)$, we have
%$$\|[b, T_R^{\delta}]h\|_{H^{1}(\rn)}\leq C \|h\|_{H_b^{\varphi}(\rn)}.$$
%\end{theorem}

\begin{proof}[\bf{Proof of Theorem \ref{Theo-main-b-3}$\rm(iii)$}]
 Assume that $h=\sum_{k=1}^L h_k$, where $h_k$  follows the same definition as in \cite[(5.14)]{F-L2024}) and it is a  multiple of $(\varphi, \infty, 0)$-atom.
Let $\Pi_{j}$ be defined in \cite[(3.3)]{F-L2024} for $j\in\{1, 2, 3, 4\}$.
By replacing the Calder\'{o}n-Zygmund  operator with the Bochner-Riesz means operator $T_R^{\delta}$ in the proof of \cite[Theorem~5.13]{F-L2024}),
we find that if there exists a positive constant $C$ such that
\begin{align}\label{eq-omega-k}
\omega_k=(b-b_{B_{k}})T_R^{\delta}h_{k}\lf(C\|h_{k}\|_{L^\infty_{B_{k}}(\rn)}\|{\mathbf 1}_{B_{k}}\|_{L^{\varphi}(\rn)}\|b\|_{\mathrm{BMO}_{\varphi}(\rn)}\r)^{-1}
\end{align}
is a $(1,\, p,\, \vec{\nu})$-molecule, then we can use
\cite[Proposition~3.15, (5.15) and (5.17)]{F-L2024} to obtain
\begin{align*}
\|\mathfrak{U}(h,\, b)\|_{H^{1}(\mathbb{R}^{n})}
\ls \|h\|_{H^{\varphi}(\mathbb{R}^{n})}\|b\|_{\mathrm{BMO}_{\varphi}(\rn)},
\end{align*}
where
\begin{align*}
\mathfrak{U}(h, b)(x)
=\sum_{k=1}^L(b(x)-b_{B_{k}})T_R^{\delta}h_{k}(x)
-T_R^{\delta}\lf(\sum_{k=1}^L\Pi_{2}(h_{k},\, b-b_{B_{k}})\r)(x)
\quad \text{for a.\,e.} \ x\in\rn.
\end{align*}
This, along with the boundedness of $T_R^{\delta}$ on $H^{1}(\mathbb{R}^{n})$
(see \cite{Stein-Taibleson-1981, P.-1981}) and \cite[Lemma~5.9 and Propositions~3.14, 3.16 and 3.17]{F-L2024}, yields
\begin{align*}
[b, T_R^{\delta}]h
=\mathfrak{U}(h,\, b)-T_R^{\delta}(\Pi_{1}(h,\, b))-T_R^{\delta}(\Pi_{3}(h,\, b))-T_R^{\delta}(\Pi_{4}(h,\, b))
\end{align*}
is bounded from $H^{\varphi}_{b}(\mathbb{R}^{n})\times\mathrm{BMO}_{\varphi}(\rn)$ to $H^1(\rn)$. This implies Theorem \ref{Theo-main-b-3}(iii).

It remains to show that the function $\omega_k$ defined in \eqref{eq-omega-k}
is a $(1,\, p,\, \vec{\nu})$-molecule.
First, for some $p_1\in(p,\infty)$,  by H\"{o}lder inequality,
the boundedness of $T_R^{\delta}$ on $L^{p_1}(\mathbb{R}^{n})$ and \eqref{eq.B-1}
we see that
\begin{align*}
\frac{\left(\int_{B_{k}}|(b(x)-b_{B_{k}})T_R^{\delta}h_{k}(x)|^{p}\,dx\right)^{1/p}}
{\|h_{k}\|_{L^\infty_{B_{k}}(\rn)}\|{\mathbf 1}_{B_{k}}\|_{L^{\varphi}(\rn)}\|b\|_{\mathrm{BMO}_{\varphi}(\rn)}}
&\leq \frac{\|T_R^{\delta}h_{k}\|_{L^{p_1}(\rn)}
\left(\int_{B_{k}}|b(x)-b_{B_{k}}|^{\frac{p_1p}{p_1-p}}\,dx\right)^{\frac{p_1-p}{p_1p}}}
{\|h_{k}\|_{L^\infty_{B_{k}}(\rn)}\|{\mathbf 1}_{B_{k}}\|_{L^{\varphi}(\rn)}\|b\|_{\mathrm{BMO}_{\varphi}(\rn)}}\\
&\ls \frac{\|h_{k}\|_{L^{p_1}(\rn)}|B_{k}|^{\frac{p_1-p}{p_1p}-1}\|{\mathbf 1}_{B_{k}}\|_{L^{\varphi}(\rn)}\|b\|_{\mathrm{BMO}_{\varphi}(\rn)}}
{\|h_{k}\|_{L^\infty_{B_{k}}(\rn)}\|{\mathbf 1}_{B_{k}}\|_{L^{\varphi}(\rn)}\|b\|_{\mathrm{BMO}_{\varphi}(\rn)}}\\
&\ls |B_{k}|^{1/p-1},
\end{align*}
which implies that $\omega_k$ satisfies Definition \ref{def.4.1}(i).
From the assumption \eqref{eq.2.36}, it follows that $\omega_k$ satisfies Definition \ref{def.4.1}(iii).

Below, we only need to verify that $\omega_k$ satisfies Definition \ref{def.4.1}(ii).
According to the condition $2n/(n+1+2\delta_0)<i(\varphi),$
we assume $\varphi$ is of lower type $\sigma$ with
\begin{align}\label{2n-siga-ivar}
2n/(n+1+2\delta_0)<\sigma<i(\varphi).
\end{align}
By this,  \eqref{eq.2.32}, $\|a_{j}\|_{L^{q}_{B}}\leq \|a_{j}\|_{L_{B_{k}}^{\infty}},$ \eqref{eq.B-1}, \eqref{eq.B-2} and \eqref{eq.simga-k1}, we obtain
\begin{align*}
&\left(\int_{(2^{j}B_{k})\setminus  (2^{j-1}B_{k})}|[b(x)-b_{B_{k}}]T_R^{\delta}h_{k}(x)|^{p}\,dx\right)^{1/p}\\
&\quad\ls
2^{-j\left(\delta_0+\frac{n+1}{2}\right)}\|h_{k}\|_{L_{B_{k}}^{p}}
\lf[\left(\int_{2^{j}B_{k}}\lf|b(x)-b_{2^{j}B_{k}}\r|^{p}\,dx\right)^{1/p}
+|2^{j}B_{k}|^{1/p}\lf|b_{2^{j}B_{k}}-b_{B_{k}}\r|\r]\\
&\quad\ls
2^{-j\left(\delta_0+\frac{n+1}{2}\right)}\|h_{k}\|_{L_{B_{k}}^{\infty}}
\|{\mathbf 1}_{2^{j}B_{k}}\|_{L^{\varphi}(\rn)}\|b\|_{\mathrm{BMO}_{\varphi}(\rn)}|2^{j}B_k|^{1/p-1}\\
&\quad\qquad+2^{-j\left(\delta_0+\frac{n+1}{2}\right)}\|h_{k}\|_{L_{B_{k}}^{\infty}}
2^{jn/\sigma}\|{\mathbf 1}_{B_{k}}\|_{L^{\varphi}(\rn)}\|b\|_{\mathrm{BMO}_{\varphi}(\rn)}|2^{j}B_k|^{1/p-1}\\
&\quad\ls
2^{-j\left(\delta_0+\frac{n+1}{2}-\frac{n}{\sigma}\right)}
\|h_{k}\|_{L_{B_{k}}^{\infty}}\|{\mathbf 1}_{B_{k}}\|_{L^{\varphi}(\rn)}\|b\|_{\mathrm{BMO}_{\varphi}(\rn)}|2^{j}B_k|^{1/p-1},
\end{align*}
which implies
$$\|\omega_{k}\mathbf 1_{(2^{j}B_{k})\setminus  (2^{j-1}B_{k})}\|_{L^p(\rn)}\le 2^{-j\left(\delta_0+\frac{n+1}{2}-\frac{n}{\sigma}\right)} |2^{j}B_{k}|^{1/p-1}.$$
Denote $\vec{\nu}=\{\nu_{j}\}_{j\in\nn}$ with $\nu_{j}=2^{-j\left(\delta_0+\frac{n+1}{2}-\frac{n}{\sigma}\right)}$ for $j\in\nn$.
By \eqref{2n-siga-ivar}, it leads to
$$\sum_{j=1}^{\infty}j\nu_{j}\ls 1.$$
Then $\omega_k$ satisfies Definition \ref{def.4.1}(ii). 
This completes the proof of Theorem \ref{Theo-main-b-3}(iii).
\end{proof}

\begin{lemma}\label{lemma4.3}
Let $q\in(1,\infty)$ and $\mathfrak{T}\in\{\mathcal{G}, \mathbf{S}, T^{\delta}_{*}\}$.
Assume that $\varphi$ is a growth function.
Then there exist a nonzero $(\varphi, q, 0)$-atom $a_{0}$ and a ball $U$ such that $U\subset\rn\setminus\operatorname{supp}a_{0}$ and
$$\mathfrak{T}(a_{0})(x)>0 \quad \text{for all} \ x\in U.$$
\end{lemma}

\begin{proof}
Choose a nonzero function $a_1\in C^{\infty}_c(\rn)$ such that $\int_{\rn}a_1(x)\,dx=0$. Then there exist $x_{1}\in\rn$ and finite value $L$ such that $\operatorname{supp}a_{1}\subset B=B(x_1, L)$ and $\|\mathbf{1}_{B(x_1, L)}\|_{L^{\varphi}(\rn)}$ is finite value.
Choose $\lambda$ satisfying
$$0<|\lambda|<\min\lf\{1, \ \lf(\|\mathbf{1}_{B(x_1, L)}\|_{L^{\varphi}(\rn)}\|a_1\|_{L^{q}_{B(x_1, L)}}\r)^{-1}\r\},$$
Taking $a_{0}:=\lambda a_1$, it follows that $a_{0}$ is a nonzero $(\varphi, q, 0)$-atom.
We now proceed to prove the target result for $\mathfrak{T}=\mathcal{G}, S, T^{\delta}_{*}$, respectively.

\medskip

{\it Step 1: There exists a ball $U$ such that $U\subset\rn\setminus\operatorname{supp}a_{0}$ and $\mathcal{G}(a_{0})(x)>0$ for all $x\in U$.}
Let $u(x,t)=P_{t}*a_{0}(x)$ with $x\in\rn$ and $t>0$.
We first prove that $\mathcal{G}(a_{0})$ cannot vanish identically on $\rn\setminus \operatorname{supp}a_{0}$. Using the proof by contradiction, we assume that
$$\mathcal{G}(a_{0})(x)=0 \quad \text{for all} \ x\in\rn\setminus \operatorname{supp}a_{0}.$$
Then $[\mathcal{G}(a_{0})]^{2}(x)=\int_{0}^{\infty}|\nabla_{u}(x,t)|^{2}t\,dt=0$, which implies that for any $x\in\rn\setminus \operatorname{supp}a_{0}$ and $ t>0$,
$$|\nabla_{u}(x,t)|=0.$$
Therefore,
$$\partial_{t}u(x,t)=0 \quad \text{for all} \ x\in\rn\setminus \operatorname{supp}a_{0} \ \text{and} \ t>0.$$
Thus, taking Fourier transform and obtaining
 $$-2\pi|\xi|e^{-2\pi t|\xi|}\widehat{a_{0}}(\xi)=0,$$
 hence,
 $$\widehat{a_{0}}(\xi)=0 \quad \text{for all} \ \xi\neq \vec{\mathbf{0}}.$$
In addition, by $\int_{\rn}a_{0}(x)\,dx=0$, we have $\widehat{a_{0}}(\vec{\mathbf{0}})=0$.
Then $\widehat{a_{0}}\equiv 0$, and hence $a_{0}\equiv 0$, which contradicts the earlier conclusion that $a_{0}$ is nonzero.
Thus, there exists $x_0\in\rn\setminus \operatorname{supp}a_{0}$ such that
$$\mathcal{G}(a_{0})(x_0)>0.$$
Consequently, there exists $t_0>0$ such that
$|\nabla_{u}(x_0,t_0)|>0.$
By continuity, there exists $\eta>0$ such that for any $|y-x_0|<\eta$ and $|t-t_0|<\eta$, $$|\nabla_{u}(y,t)|>0.$$
Taking $\eta$ smaller if necessary, we can also assume that $B(x_0, \eta)\cap \operatorname{supp}a_{0}=\emptyset$ and $0<\eta<t_0$. Set $U:=B(x_0, \eta/4)$, then $U\subset \rn\setminus \operatorname{supp}a_{0}$.
Thus,
  $$\mathcal{G}(a_{0})(x)=\lf(\int_{0}^{\infty}|\nabla_{u}(x,t)|^{2}t\,dt\r)^{1/2}>0 \quad \text{for all} \ x\in U.$$

  \medskip

{\it Step 2: Verifying $\mathbf{S}(a_{0})(x)>0$ for all $x\in U$.} We continue to adopt the parameters $\eta>0$, $t_0>0$  and ball $U:=B(x_0, \eta/4)$ selected in {\it Step 1}, which satisfy $0<\eta<t_0$ and
$U\subset \rn\setminus \operatorname{supp}a_{0}$.
For any $x, y\in U=B(x_0, \eta/4)$ and $t\in(t_0-\eta/2, t_0+\eta/2)$, we have
$$|x-y|\leq |x-x_0|+|y-x_0|<\eta/2<t.$$
This implies  the set
$$\lf\{(y, t): \, y\in U=B(x_0, \eta/4) \ \text{and} \ t\in(t_0-\eta/2, t_0+\eta/2)\r\}$$
is contained in the cone $\Gamma(x)$. Then $|\nabla u_{a_{0}}(y,t)|>0$ for all $y\in U=B(x_0, \eta/4)$ and $t\in(t_0-\eta/2, t_0+\eta/2)$, hence,
$$\mathbf{S}(a_{0})(x)>0 \quad \text{for all} \  x\in U.$$

\medskip

{\it Step 3: Verifying $T^{\delta}_{*}(a_{0})(x)>0$ for all $x\in U$.}
Using the contradiction argument, we suppose that
$$T^{\delta}_{*}(a_{0})(x)=0 \quad \text{for all} \ x\in\rn\setminus\operatorname{supp}a_{0}.$$
Then for every $R>0$,
$$T^{\delta}_R(a_{0})(x)=0 \quad \text{for all} \ x\in\rn\setminus\operatorname{supp}a_{0}.$$
So, taking Fourier transform on the two sides of the preceding formula, we obtain 
$$\left(1-\frac{|\xi|^2}{R^2}\right)^{\delta}_{+}\widehat{a_{0}}(\xi)=0.$$
Thus,
$$\widehat{a_{0}}(\xi)=0 \quad \text{for all} \  |\xi|<R.$$
By the arbitrary property of $R$, we obtain $\widehat{a_{0}}\equiv 0$. Hence $a_{0}=0$, which contradicts the aforementioned conclusion that $a_{0}$ is nonzero.
So, there exists $x_0\in \rn\setminus \operatorname{supp}a_{0}$ such that $T^{\delta}_{*}a_{0}(x_0)>0$. From the
definition of $T^{\delta}_{*}$, it follows that there exists $R_0>0$ such that
$$|T^{\delta}_{R_0}a_{0}(x_0)|>0.$$
We keep using the previously chosen parameter $\eta>0$  and ball $U:=B(x_0, \eta/4)$ from {\it Step 1}, with the constraint $U\subset \rn\setminus \operatorname{supp}a_{0}$.
Thus,
$$|T^{\delta}_{R_0}a_{0}(x)|>0 \quad \text{for all} \  x\in U.$$
Consequently,
$$T^{\delta}_{*}a_{0}(x)>0 \quad \text{for all} \  x\in U.$$

Combining the discussions from {\it Step 1} to {\it Step 3}, we obtain the desired conclusion.
\end{proof}

\begin{lemma}\label{lemma4.4}
Let $q\in(1,\infty)$ and $\varphi$ be a growth function.
Assume that the function $\Omega\in L^1(S^{n-1})$ satisfies
$$\int_{S^{n-1}}\Omega(\theta)\,d\sigma(\theta)=0 \quad \text{and} \quad  \Omega\not\equiv0.$$
Then there exist a nonzero $(\varphi, q, 0)$-atom $a_{0}$ and a ball $U$ such that $U\subset\rn\setminus\operatorname{supp}a_{0}$ and
$$\mu_\Omega(a_{0})(x)>0 \quad \text{for all} \  x\in U.$$
\end{lemma}

\begin{proof}
Denote $K(z)=\frac{\Omega(z)}{|z|^{n-1}}$ for $z\in\rn\setminus\{\vec{\mathbf{0}}\}$.
Thus, by the condition $\Omega\in L^1(S^{n-1})$ and $\Omega\not\equiv 0$, we have $K\in L_{loc}^{1}(\rn\setminus\{\vec{\mathbf{0}}\})$ and $K\not\equiv 0$.
Then there exists a function
$\psi\in C_c^{\infty}(\rn\setminus\{\vec{\mathbf{0}}\})$ such that
$$\operatorname{supp}\psi\subset\{z\in\rn:\ 0<r_1<|z|<r_2<\infty\}$$
and
\begin{align}\label{Kz-neq0}
\int_{\rn}K(z)\psi(z)\,dz\neq 0.
\end{align}
Further, choose $\phi\in C_c^{\infty}(\{y\in\rn: |y|>3r_2\})$ such that
$$\int_{\rn}\phi(y)\,dy=\int_{\rn}\psi(y)\,dy.$$
Define
$$a_1(y)=\psi(-y)-\phi(y) \quad \text{for all} \ y\in\rn.$$
Then $a_1\in C_c^{\infty}(\rn)$, $\int_{\rn}a_1(y)\,dy=0$ and $a_1\not\equiv 0$.
Thus, there exists a finite ball $B$ such that $\operatorname{supp}a_1\subset B$ and
$\|\mathbf{1}_{B}\|_{L^{\varphi}(\rn)}$ is finite value. Take
$$a_{0}=\lambda a_1,$$
where $\lambda$ satisfies
$$0<|\lambda|\leq \frac{\|\mathbf{1}_B\|^{-1}_{L^{\varphi}(\rn)}}{\|a_1\|_{L^{q}_B}}<\infty.$$
So, $a_{0}$ is a nonzero $(\varphi, q, 0)$-atom.

Below, we verify that there exists a ball $U$ such that $U\subset\rn\setminus\operatorname{supp}a_{0}$ and $\mu_\Omega(a_{0})(x)>0$ for all $x\in U$.  Define
$$F_t(x):=\int_{|x-y|\leq t}K(x-y)a_{0}(y)\,dy \quad \text{for} \ t>0 \ \text{and} \ x\in\rn.$$
Denote $\tilde{\psi}(\cdot)=\psi(-\cdot)$ and $t_0=2r_2$.
If $y\in\operatorname{supp}\tilde{\psi}$, then $|y|\leq r_2<t_0$.
However, we see that $|y|>3r_2>t_0$ for all $y\in \operatorname{supp}\phi$.
By this and \eqref{Kz-neq0},  it follows that
\begin{align*}
F_{t_0}(0)=\int_{|y|\leq t_0}K(-y)a_{0}(y)\,dy =\lambda \int_{|y|\leq t_0}K(-y)\psi(-y)\,dy
=\lambda \int_{|z|\leq t_0}K(z)\psi(z)\,dz\neq0.
\end{align*}
Define
$$G(x)=\int_{\rn}K(x-y)\psi(-y)\,dy  \quad \text{for all} \ x\in\rn.$$
Due to $\operatorname{supp}\tilde{\psi}\subset\{y:r_1\leq|y|\leq r_2\}$, we obtain
that $|x-y|\geq |y|-|x|>r_1/2>0$ for all $|x|<r_1/2$.
Since $G(x)=K*\hat{\psi}(x)$, $\hat{\psi}=\psi(-y)\in C_c^{\infty}(\rn\setminus\{\vec{\mathbf{0}}\})$
and $K\in L^{1}_{loc}(\rn\setminus\{\vec{\mathbf{0}}\}$), we deduce that $G(x)\in C^{\infty}(\rn)$ and $G(x)$ is continuous on $B\lf(0, \frac{r_1}{4}\r)$.
Moreover, it follows from \eqref{Kz-neq0} that
$$G(0)=\int_{\rn}K(-y)\psi(-y)\,dy\neq0,$$
which implies that there exists $\eta>0$ such that
$$G(x)\neq0 \quad \text{for all} \  x\in B(0, \eta).$$
Take $\delta=\min\{r_1/4, \eta/2\}$ and $U=B(0, \delta)$. Then $U\cap\operatorname{supp}a_{0}=\emptyset$ and $t_1:=t_0+\delta=2r_2+\delta$.
For any $x\in U$ and any $t\in[t_0, t_1]$, we have
\begin{align*}
\begin{cases}
|x-y|\leq |x|+|y|<\delta+r_2\leq t_0\leq t &\text{as}\ y\in\operatorname{supp}\tilde{\psi};\\
|x-y|\geq |y|-|x|>3r_2-\delta>2r_2+\delta=t_1\geq t  &\text{as}\ y\in \operatorname{supp}\phi,
\end{cases}
\end{align*}
which implies that $\operatorname{supp}\tilde{\psi}\subset\{y\in\rn: \, |x-y|\leq t\}$
and $\operatorname{supp}\phi\cap\{y\in\rn: \, |x-y|\leq t\}=\emptyset$.
Therefore, we conclude that for any $x\in U\subset B(0, \eta)$ and any $t\in[t_0, t_1]$,
$$F_t(x)=\lambda\int_{|x-y|\leq t}K(x-y)\psi(-y)\,dy=\lambda\int_{\rn}K(x-y)\psi(-y)\,dy=\lambda G(x)\neq 0.$$
This allows us to derive
$$\mu_\Omega(a_{0})(x)\geq\left(\int_{t_0}^{t_1}|F_t(x)|^{2}\,\frac{dt}{t^3}\right)^{1/2}=|\lambda G(x)|\left(\int_{t_0}^{t_1}\,\frac{dt}{t^3}\right)^{1/2}>0.$$
This completes the proof of Lemma \ref{lemma4.4}.
\end{proof}

The following theorem presents a counterexample, which illustrates that the commutators corresponding to the Lusin area integral, g-function, Marcinkiewicz integral and Bochner Riesz maximal operator are not bounded from $H_b^{\varphi}(\rn)$ to $H^1(\rn)$.

\begin{theorem}\label{the-ex-4.1}
Let $\mathfrak{T}\in\{\mathbf{S},\, g,\, \mu_\Omega,\, T^{\delta}_{*}\}$. Assume that growth function $\varphi$ satisfies $\varphi(1)\neq 0$ and
\begin{align}\label{eq-var-max}
1>i(\varphi)>\max\lf\{\frac{n}{n+1}, \,\frac{n}{n+\epsilon}, \,\frac{2n}{n+1+2\delta_0}\r\},
\end{align}
where the definitions of $\epsilon$ and $\delta_0$ coincide with the corresponding ones in Proposition \ref{proposition-mar} and Proposition \ref{proposition-BR}.
 If $a$ is a bounded nonzero $(\varphi,q,0)$-atom and $U\subset\rn\setminus\operatorname{supp}a$ is a ball such that
$$\mathfrak{T}(a)(x)>0 \quad \text{for all} \ x\in U.$$
Then there exists $b\in \mathrm{BMO}_{\varphi}(\rn)$ such that $a\in H_{b}^{\varphi}(\rn)$ and $[b,\mathfrak{T}]a\notin H^{1}(\rn)$.
\end{theorem}

\begin{proof} The verification of this conclusion will be carried out in three distinct steps.

\medskip

{\it Step 1: If $\varphi$ satisfies \eqref{eq-var-max}, then there exists function $b$ such that $b\in\mathrm{BMO}_{\varphi}(\rn)$.}
Based on the condition that growth function $\varphi$ satisfies \eqref{eq-var-max},
we assume that $\varphi$ is of upper type $1$ and lower type $\varrho$ such that $$\max\lf\{\frac{n}{n+1},\,\frac{n}{n+\epsilon},\,\frac{2n}{n+1+2\delta_0}\r\}<\varrho<i(\varphi).$$
Then there exists constant $C_{\varrho}>1$ such that
for any $\lambda\in(0, 1)$,
$$\varphi(1)\leq C_{\varrho}\lambda^{\varrho}\varphi\lf(\frac{1}{\lambda}\r).$$
Therefore,
\begin{align}\label{eq-var-x0-1}
\varphi\lf(\frac{1}{\lambda_{0}}\r)|B(x_0,\, 1)|\geq \lambda_{0}^{-\varrho}\varphi(1)C^{-1}_{\varrho}|B(x_0,\, 1)|>1
\end{align}
with
$$\lambda_{0}=\min\lf\{\frac{1}{2}, \ \frac{\lf(\varphi(1)C^{-1}_{\varrho}|B(x_0,\, 1)|\r)^{1/\varrho}}{2}\r\}>0.$$
From \eqref{eq-var-x0-1} and the definition of $\|{\mathbf 1}_{B(x_0,\, 1)}\|_{L^{\varphi}(\rn)}$, it follows that
$$\|{\mathbf 1}_{B(x_0,\, 1)}\|_{L^{\varphi}(\rn)}>\lambda_{0}.$$
This, along with \cite[Lemma~2.4]{Z-F-C-2026}, yields
\begin{align}\label{eq.4.40-1-1}
\frac{1}{\|{\mathbf 1}_{B(x_0,\, r)}\|_{L^{\varphi}(\rn)}}
&\ls
\begin{cases}
\frac{1}{r^{n}\|{\mathbf 1}_{B(x_0,\, 1)}\|_{L^{\varphi}(\rn)}} &\text{as}\ r\in[1, \infty);\\
\frac{1}{r^{n/\varrho}\|{\mathbf 1}_{B(x_0,\, 1)}\|_{L^{\varphi}(\rn)}}  &\text{as}\ r\in(0, 1)
\end{cases}\\
&\ls
\begin{cases}
\frac{1}{r^{n}\lambda_{0}} &\text{as}\ r\in[1, \infty);\\
\frac{1}{r^{n/\varrho}\lambda_{0}}  &\text{as}\ r\in(0, 1).
\end{cases}\notag
\end{align}

Choose $0\leq b\in C^{\infty}_{c}(U)$ satisfying $b\not\equiv 0$, where $U$ is a ball.
Fix a ball $B=B(x_0, r)$, if $0<r\leq 1$, then we apply \eqref{eq.4.40-1-1} and
$$\varrho>\max\lf\{\frac{n}{n+1},\,\frac{n}{n+\epsilon},\,\frac{2n}{n+1+2\delta_0}\r\}>\frac{n}{n+1}$$
to  derive
\begin{align*}
\frac{1}{\|\mathbf{1}_{B}\|_{L^{\varphi}(\rn)}}\int_B|b(x)-b_B|\,dx
&= \frac{1}{\|\mathbf{1}_{B}\|_{L^{\varphi}(\rn)}}\int_B\lf|b(x)-\frac{1}{|B|}
\int_{B}b(y)\,dy\r|\,dx \\
&\leq \frac{1}{r^{n/\varrho}\lambda_{0}}\frac{1}{|B|}\int_{B\cap U}\int_{B\cap U}|b(x)-b(y)|\,dy\,dx\\
&\leq \frac{1}{r^{n/\varrho}\lambda_{0}}\frac{L}{|B|}\int_B\int_{B}|x-y|\,dy\,dx\\
&\leq \frac{1}{r^{n/\varrho}\lambda_{0}}2Lr|B|\ls r^{n+1-n/\varrho}L\ls L,
\end{align*}
where the three step used $\operatorname{supp}b\subset U$ and $b\in C^{\infty}_{c}(U)$, that is,
$$|b(x)-b(y)|\leq L|x-y| \quad \text{for} \ x, y\in U.$$
If $r>1$, by \eqref{eq.4.40-1-1} and
$$|b(x)-b_B|\leq |b(x)|+|b_B|\leq |b(x)|+\frac{1}{|B|}\int_{B}|b(y)|\,dy,$$
we obtain 
\begin{align*}
\frac{1}{\|\mathbf{1}_{B}\|_{L^{\varphi}(\rn)}}\int_B|b(x)-b_B|\,dx
\leq
\frac{2}{\|\mathbf{1}_{B}\|_{L^{\varphi}(\rn)}}\int_B|b(x)|\,dx
\ls
\frac{2\|b\|_{L^1(\rn)}}{r^{n}\lambda_{0}}
<\infty.
\end{align*}
Combining the discussions above, we conclude that $b\in \mathrm{BMO}_{\varphi}(\rn)$.
\medskip

{\it Step 2: Verifying $a\in H_b^{\varphi}(\rn)$.}
Since $\operatorname{supp}b\subset U$ and $U\cap\operatorname{supp}a=\emptyset$, we have
\begin{align*}
b(y)=0 \quad \text{for all} \ y\in\operatorname{supp}a,
\end{align*}
where implies
\begin{align*}
[b,\mathfrak{M}]a(x)=\mathfrak{M}((b(x)-b(\cdot))a(\cdot))(x)=|b(x)|\mathfrak{M}a(x).
\end{align*}
According to the fact that $\operatorname{supp}b$ is compact, we deduce that
$|x-y|>0$ for all $x\in \operatorname{supp}b$ and $y\in\operatorname{supp}a$.
If we claim
\begin{align}\label{G-xyt}
G(x, y, t):=t^{-n}\left(1+|x-y|/t\right)^{-(n+1)}\leq |x-y|^{-n} \quad \text{for all} \ x\in \operatorname{supp}b, \, y\in\operatorname{supp}a, \, t>0,
\end{align}
then,  by the definition of $\mathcal{S}_{0}(\rn)$, we see that for any $\psi\in\mathcal{S}_{0}(\rn)$,
\begin{align*}
|\psi_{t}*a(x)|
&\leq \int_{\operatorname{supp}a}|\psi_{t}(x-y)||a(y)|\,dy\\
&\ls \|a\|_{L^1(\rn)}t^{-n}\left(1+|x-y|/t\right)^{-(n+1)}\\
&\ls \|a\|_{L^1(\rn)}|x-y|^{-n}.
\end{align*}
This implies that $a\in H_b^{\varphi}(\rn)$.

It remains to show \eqref{G-xyt}. 
If $0<t\leq |x-y|$, then
\begin{align*}
G(x, y, t)
\leq t^{-n}(|x-y|/t)^{-(n+1)}=t|x-y|^{-(n+1)}\leq |x-y|^{-n}.
\end{align*}
If $t>|x-y|$, then
\begin{align*}
G(x, y, t)
\leq t^{-n}\leq |x-y|^{-n}.
\end{align*}
Thus, \eqref{G-xyt} holds.

\medskip
{\it Step 3: Verifying $[b,\mathfrak{T}]a\notin H^{1}(\rn)$.}
If $[b,\mathfrak{T}]a\in H^1(\rn)$, then the vanishing property of $H^1(\rn)$ allows us obtain
\begin{align}\label{var-0-0}
\int_{\rn}[b,\mathfrak{T}]a(x)\,dx=0.
\end{align}
From the discussion of {\it Step 2},  it follows that $b(y)=0$ for $y\in\operatorname{supp}a$,
which implies
$$[b,\mathfrak{T}]a(x)=b(x)\mathfrak{T}(a)(x).$$
This, along with $0\leq b(x)\not\equiv 0$ and
$$\mathfrak{T}(a)(x)>0 \quad \text{for all} \ x\in U,$$
yields
$$\int_{\rn}[b,\mathfrak{T}]a(x)\,dx=\int_{\rn}b(x)\mathfrak{T}(a)(x)\,dx>0,$$
which contradicts \eqref{var-0-0}, so the hypothesis is invalid, and hence
$$[b,\mathfrak{T}]a\notin H^1(\rn).$$
We finish the proof of Theorem \ref{the-ex-4.1}.
\end{proof}

\end{document}